\documentclass[11pt]{article}
\usepackage{amssymb} %for Blackboard bold etc
\usepackage{graphicx} %for including eps graphics
\usepackage{amsthm} %for defining theorems
\usepackage{color}
\usepackage{amsmath}
\usepackage{mathrsfs} 
\usepackage{mathtools} 
\usepackage{hyperref}
\usepackage{dsfont}
\usepackage[parfill]{parskip}
\usepackage[a4paper, total={6in, 8in}]{geometry}
\usepackage{bbm}
\usepackage{cleveref} % Must be last referencing package loaded
\usepackage{enumerate}
\usepackage[normalem]{ulem} % crossout!

\usepackage{setspace}
\newcommand{\Palm}{\scriptscriptstyle(0,T_0)}

\newcommand{\mathd}{\mathrm{d}}

\newcommand{\R}{\mathbb R}

\newcommand{\E}{\mathbb E}
\renewcommand{\P}{\mathbb P}

\renewcommand{\phi}{\varphi}

\newcommand{\Z}{\mathbb Z}
\newcommand{\N}{\mathbb{N}}
\renewcommand{\P}{\mathbb P}
\newcommand{\X}{\mathcal X}

\newcommand{\V}{\mathcal V}

\newcommand{\0}{\mathbf{0}}
\newcommand{\x}{\mathbf{x}}
\newcommand{\y}{\mathbf{y}}
\newcommand{\z}{\mathbf{z}}
\renewcommand{\v}{\mathbf{v}}

\newcommand{\e}{\mathbf{e}}

\newcommand{\abs}[1]{\left| #1 \right|}

\newcommand{\set}[1]{\left\{ #1 \right\}}

\newcommand{\gaus}[1]{\lfloor #1 \rfloor}

\newcommand{\conn}{\leftrightarrow}
\newcommand{\1}{\mathbf{1}}
\newcommand{\st}{\star}
\renewcommand{\sp}{\circledast}

\DeclareMathOperator{\St}{Star}

\title{The contact process on scale-free geometric random graphs} 
\author{Peter Gracar$^\dag$ and Arne Grauer$^\ddag$\footnote{Corresponding author, E-Mail: arnegrauer@posteo.de}\\
{\sl\small $^\dag$School of Mathematics, University of Leeds, United Kingdom}\\
{\sl\small $^\ddag$Mathematisches Institut, Universit\"at zu K\"oln, Germany}
}
\date{}

\usepackage[style=numeric, backend = biber, url = false, doi = false, isbn = false, maxnames = 10]{biblatex}
\usepackage{todonotes}

\theoremstyle{plain} %theorems
\newtheorem{theorem}{Theorem}[section]
\newtheorem{prop}[theorem]{Proposition}
\newtheorem{lemma}[theorem]{Lemma}

\newtheorem{assum}{Assumption}[section]

\theoremstyle{definition} %definitions

\theoremstyle{remark} %notes and remarks
\newtheorem{remark}{Remark}[section]

\begin{document}

\maketitle

\begin{abstract}
We study the contact process on a class of geometric random graphs with scale-free degree distribution, defined on a Poisson point process on $\R^d$. This class includes the \emph{age-dependent random connection model} and the \emph{soft Boolean model}. In the ultrasmall regime of these random graphs we provide exact asymptotics for the non-extinction probability when the rate of infection spread is small and show for a finite version of these graphs that the extinction time is of exponential order in the size of the graph.
\end{abstract}

\section{Introduction}
In recent years the contact process has been studied extensively as a simple model for the spread of infection in a population or on a network. In this model each vertex of a given locally-finite graph has one of two states, $0$ or $1$ for each $t\geq 0$, indicating whether the vertex is healthy or infected at time $t$. An infected vertex \emph{recovers} at rate $1$ and transmits the infection to a neighbouring vertex with rate $\lambda>0$, independently of everything else. Precisely, the contact process on a locally-finite graph $G=(V,E)$ is a continuous-time Markov process $(\xi_t)_{t\geq 0}$ on the space $\{0,1\}^V$. By identifying $\xi_t$ with the subset $\{x\in V:\xi_t(x)=1\}\subset V$ for each $t\geq 0$ the transition rates are given by
\begin{align}
\begin{split}
\xi_t \to \xi_t\backslash \{x\}\quad &\text{for}\ x\in \xi_t\ \text{at rate}\ 1\text{, and}\\
\xi_t \to \xi_t\cup \{x\}\quad &\text{for}\ x\notin \xi_t\ \text{at rate}\ \lambda\big\vert\{y\in \xi_t: x\sim y\}\big\vert,
\end{split}\label{eq:transrate}
\end{align}
where we denote by $x\sim y$ that $x$ and $y$ are connected by an edge.
%We write $(\xi^A_t)_{t\geq 0}$ for the contact process with \emph{initial condition} $A\subset V$, i.e. $\xi_0^A = A$ and $(\xi^\x_t)_{t\geq 0}$ if $A = \set{\x}$.

Note that the contact process has a single absorbing state, corresponding to the configuration where all vertices are healthy. Thus, a natural question on the behaviour of the contact process is whether this state is reached in finite time, i.e. whether the \emph{extinction time} of the contact process on $G$, defined by
\[
\tau_G := \inf\{t>0:\xi_t = \emptyset\},
\]
is finite. 

On the lattice $\Z^d$ there exists a critical value $\lambda_c(\Z^d)$ exhibiting a phase transition in whether the process dies out almost surely or not. If $\lambda\leq \lambda_c(\Z^d)$, the extinction time $\tau_{\Z^d}$ is almost surely finite for any initial configuration where only finitely many sites are infected and we say it dies out, whereas if $\lambda>\lambda_c(\Z^d)$ there is a positive probability that the extinction time is infinite even if the infection only starts in a single vertex, see \cite{Lig1999}.
On finite graphs the extinction time is always almost surely finite and a more natural question in this setting is to ask how long the infection survives until it reaches the absorbing state. Interestingly, on the restriction of $\Z^d$ to finite boxes there is again a phase transition with respect to the critical value $\lambda_c(\mathbb{Z}^d)$. If $\lambda< \lambda_c(\Z^d)$ the extinction time is of logarithmic order in the volume of the box whereas if $\lambda>\lambda_c(\Z^d)$ the contact process survives much longer and the extinction time is of exponential order in the volume of the box. In the latter case the contact process is said to be in a \emph{metastable} situation where it stabilizes for an exponentially long amount of time before it reaches the absorbing state where all vertices are healthy, see \cite{Lig1999} for further details. The behaviour of the extinction time of the contact process has been studied on various different finite graphs including on finite regular trees by Stacey \cite{Sta2001}, Cranston et al. \cite{CraMMV2014} and Mountford et al. \cite{MouMVY2016}, on regular graphs by Lalley and Su \cite{Lalley2017} and Mourrat and Valesin \cite{MouV2016}, on Erdós-Renyi graphs by Bhamidi et al. \cite{BhaNNS2021} and for a general large class of finite graphs by Mountford et al. \cite{MouMVY2016} and Schapira and Valesin \cite{SchV2021}.

The situation changes dramatically if we consider random graphs with a scale-free degree distribution such as the configuration model or preferential attachment networks. On these graphs the critical value of $\lambda$ is zero, therefore for any choice of $\lambda>0$ the extinction time is of exponential order in the size of the graph, see \cite{ChaD2009} and \cite{MouMVY2016} for the configuration model and \cite{BerBCS2005} for preferential attachment networks. For these models further results on the metastability are given by Mountford et al. \cite{MouVY2013} on the configuration model, where they provide estimates on the rate of decay of the density of infected vertices in terms of $\lambda$ at a time when the infection has not yet reached the absorbing states, for $\lambda$ small and graph large enough. This rate of decay solely depends on the power-law exponent $\tau$ of the scale-free degree distribution\footnote{We say a graph has a scale-free degree distribution with power-law exponent $\tau$, if the degree distribution of a typical vertex $\mu$ satisfies $\mu(k) = k^{-\tau+o(1)}$ when $k$ is large.}, precisely it is given by
\begin{align}
\rho_\tau(\lambda) = \begin{cases} \lambda^{1/(3-\tau)}\quad &\text{if}\ \tau \in (2,\frac{5}{2}]\\ \frac{\lambda^{2\tau-3}}{\log(1/\lambda)^{\tau-2}}\quad &\text{if}\ \tau \in (\frac{5}{2},3]\\ \frac{\lambda^{2\tau-3}}{\log(1/\lambda)^{2\tau-4}}\quad &\text{if}\ \tau \in (3,\infty)\end{cases}.\label{eq:product-rates}
\end{align}
This result seems to have a universal character as the same rate of decay has been shown for hyperbolic random graphs by Linker et al. \cite{LinMSV2021}. The latter can be seen as a geometric variant of the configuration model, as in both cases the probability to form an edge between two given vertices depends on the product of independent weights which are  assigned to each vertex in the graph. For further results on the density of infected vertices at a time when the contact process is still alive on the configuration model with $\tau \in (1,2]$ and preferential attachment networks see \cite{CanS2015} and \cite{Can2017}. To obtain these estimates the analysis of the contact process on the corresponding limit graphs is important, i.e. the corresponding Galton-Watson process for the configuration model and the infinite hyperbolic model. In fact, for these models the rate of decay given by $\rho_\tau$ coincides with the rate of decay of the probability that the extinction time is infinite on the limit graphs when $\lambda$ goes to zero. Thus, the study of the non-extinction probability is a crucial step to obtain the stated metastability results. 

In this work we will study the contact process on a large class of geometric random graphs on a Poisson point process on $\R^d$, which have a scale-free degree distribution and allow the occurence of edges which span a large distance between two vertices, see Section \ref{subsec:framework} for a formal definition. Two motivating examples of this class, which are shortly introduced in the following,  are a soft version of the scale-free Boolean model introduced in \cite{GraGM2022} and the \emph{age-dependent random connection model} introduced in \cite{GraGLM2019} which emerges as a limit graph of a spatial preferential attachment network. 

\paragraph{The soft Boolean model}
In the Boolean model on a Poisson point process on $\R^d$ each vertex $x$ carries an independent, identically distributed random radius $R_x$, which we assume to be heavy-tailed, i.e. there exists $\gamma\in (0,1)$ such that 
\[
\P(R_x>r) \asymp r^{-d/\gamma}\ \text{as}\ r\to \infty.
\]
In the \emph{hard} version of the model two vertices are connected by an edge if the balls centered at the vertices locations with associated radii intersect. We consider a \emph{soft} version of this model, where an independent, identically distributed random variable $X(x,y)$ is assigned to each unordered pair of vertices $\{x,y\}$. Then, an edge is formed between two vertices $x$ and $y$ if and only if
\[
\frac{\abs{x-y}}{R_x+R_y}\leq X(x,y).
\]
The choice $X(x,y)=1$, for all pairs $\{x,y\}$, corresponds to the hard version of the model. Writing $X$ for an independent identically distributed copy of $X(x,y)$, we assume $X$ to be heavy-tailed with decay
\[
\P(X>r) \asymp r^{-\delta d}\ \text{as}\ r\to \infty,
\]
for some $\delta>1$, leading to a relaxation of the condition to form an edge which can be interpreted in the following way. For any pair of vertices $x$ and $y$ we take a copy of their corresponding balls and expand those by multiplying both radii with $X(x,y)$. Then, we form an edge between the two vertices if and only if the expanded balls intersect.

\paragraph{The age-dependent random connection model}
In the age-dependent random connection model each vertex carries a uniform on $(0,1)$ distributed birth time and two vertices $x$ and $y$ with birth times $t$ and $s$ are connected by an edge independently with probability
\begin{equation}
\varphi\big(\beta^{-1}(t\wedge s)^\gamma (t\vee s)^{1-\gamma}\abs{x-y}^d\big), \label{eq:conn_prob}
\end{equation}
where $\gamma\in (0,1)$, $\beta>0$ and $\varphi:(0,\infty)\to [0,1]$ is a non-decreasing function which we assume to satisfy $\varphi(r) \asymp r^{-\delta}$ as $r\to \infty$. This model emerges as a limit graph of a rescaled version of the \emph{age-based preferential attachment model} introduced in \cite{GraGLM2019}. In this model vertices are added to the graph at rate of a Poisson process with unit intensity and placed on a torus of width one. A new vertex $x$ added at time $t$ forms an edge to each existing vertex $y$ with probability given by \eqref{eq:conn_prob}, where $s$ is the time vertex $y$ has been added to the graph. As $(t/s)^\gamma$ is the asymptotic order of the expected degree at time $t$ of a vertex with birth time $s$ the age-based preferential attachment model mimics the behaviour of spatial preferential attachment networks introduced in \cite{JacM2015}.

This class of geometric random graphs has been studied recently and  exhibits a phase transition in the parameters $\gamma$ and $\delta$ such that these graphs are ultrasmall, i.e. two very distant vertices have a graph distance of doubly logarithmic order of their Euclidean distance, if and only if $\gamma>\frac{\delta}{\delta+1}$, as shown in \cite{GraGM2022}. The same regime boundary depending on the parameters $\gamma$ and $\delta$ is shown to appear in the existence of a subcritical percolation phase by Gracar et al. \cite{GraGLM2019}, giving a hint of a universal behaviour of these geometric random graphs that is remarkably different to the behaviour of spatial graph models investigated in \cite{DeivHH2013}. In this work we consider the class of geometric random graphs in their ultrasmall regime $\gamma>\frac{\delta}{\delta+1}$ and in Section \ref{sec:non-ext_prob} we study the probability $\Gamma(\lambda)$ that the contact process starting in a typical vertex does not go extinct. We prove that the critical value $\lambda_c$ is zero for these graphs and we give exact asymptotics on its rate of decay, when $\lambda$ is small, see Theorem \ref{thm:non-ext_prob}. In Section \ref{sec:exp_ext_time} we study a restriction of the these graphs to boxes $[-\frac{n^{1/d}}{2},\frac{n^{1/d}}{2}]^d$, still assuming that $\gamma>\frac{\delta}{\delta+1}$, and show that the extinction time exhibits no phase transition, i.e. for any $\lambda>0$ the extinction time is of exponential order in the volume of the boxes, see Theorem \ref{thm:exp_extinction_time}.
%See \cite{CanS} for further results in this direction on the configuration model when $\tau \in (1,2]$ and \cite{Can} for results on preferential attachment networks if $\tau >3$. The results on 

\subsection{Framework}
\label{subsec:framework}
Let $\mathscr{G}$ be a general geometric random graph on a vertex set given by a Poisson point process $\mathcal{X}$ of unit intensity on $\mathbb{R}^d\times (0,1)$. We write $\x=(x,t)$ for a vertex of the graph where we refer to $x$ as the \emph{location} and $t$ as the \emph{mark} of the vertex $\x$. For $\x_1,\ldots,\x_n\in \mathbb{R}^d\times (0,1)$, denote by $\P_{\x_1,\ldots,\x_n}$ the law of $\mathscr{G}$ given the event that $\x_1,\ldots,\x_n$ are points of the Poisson process. We consider geometric random graphs which satisfy the following assumption on the probability of the occurence of edges in the graph. This assumption is given in terms of two parameters $\gamma\in (0,1)$ and $\delta>1$. For two vertices $\x,\y\in \mathcal{X}$ we write $\x\sim \y$ if there exists an edge between them.

\begin{assum}\label{ass:main}
Given $\mathcal{X}$, edges are drawn independently of each other and there exist $\alpha,\kappa_1,\kappa_2>0$ such that, for every pair of vertices $\x=(x,t), \y=(y,s)\in \mathcal{X}$, it holds that
\[
\alpha\, \big(1\wedge \kappa_1 \, (t \wedge s)^{-\delta\gamma} \abs{x-y}^{-\delta d}\big) \leq \P_{\x,\y}(\x\sim \y)\leq \kappa_2 \, (t\wedge s)^{-\delta\gamma} (t\vee s)^{\delta(\gamma-1)}\abs{x-y}^{-\delta d}.
\]
\end{assum}

In some parts of this paper we will consider the Palm-version of $\mathscr{G}$. More precisely, we add to $\mathcal{X}$ a vertex $(0,T_0)$, where $T_0$ is an independent on the interval $(0,1)$ uniformly distributed random variable and denote by $\mathscr{G}_{\Palm}$ the resulting graph on $\mathcal{X}\cup \{(0,T_0)\}$ determined by the connection rules satisfying Assumption \ref{ass:main}. We denote the law of $\mathscr{G}_{\Palm}$ by $\P_{\Palm}$ and since $T_0$ is independent of the underlying Poisson point process $\mathcal{X}$, it holds $\P_{\Palm} = \int_0^1\P_{(0,t_0)}\mathd t_0$. We can think of $\P_{\Palm}$ of the law of $\mathscr{G}$ conditioned on the existence of a typical vertex, i.e. a vertex with typical mark, at the origin.

The parameter $\gamma$ in Assumption \ref{ass:main} describes the influence of the vertices' marks on the connection probability and determines the degree distribution of a typical vertex in $\mathscr{G}$. As edges are drawn independently of each other, the degree of a given vertex with mark $t$ is Poisson distributed with parameter $\Lambda(t)$, where by Assumption \ref{ass:main} there exist $c,C>0$ such that it holds $ct^{-\gamma}\leq \Lambda(t)\leq Ct^{-\gamma}$ for all $t\in (0,1)$. It is easy to see that the degree of a typical vertex $(0,T_0)$ is scale-free with power-law exponent $\tau = 1 + \frac{1}{\gamma}$, see \cite{GraGLM2019}. The parameter $\delta$ controls the occurence of long edges in the graph, where small values of $\delta$ lead to more long edges. A natural example of such geometric random graphs is the weight-dependent random connection model, introduced in \cite{GraHMM2022}. One can understand Assumption \ref{ass:main} in the way that graphs satisfying the assumption are dominated by the weight-dependent random connection model with preferential-attachment kernel and dominate the model with the min kernel. 

\begin{remark}
	As we will see later, each of the individual proofs depends only on one of the two inequalities from Assumption \ref{ass:main}. One could therefore separate this assumption into two separate assumptions and therefore make the statements of individual propositions used to prove the main results could be made stronger by assuming only the relevant bound from Assumption \ref{ass:main} holds. For simplicity and to help with the heuristic understanding of the proofs we omit this distinction.
\end{remark}

\subsection{The contact process and its graphical representation}
The contact process on an arbitrary locally-finite graph $G=(V,E)$ with parameter $\lambda$ is a continuous time Markov process $(\xi_t)_{t\geq 0}$ on the space $\{0,1\}^V$. At time $t$ we say a vertex $x\in V$ is infected if $\xi_t(x) = 1$ and healthy if $\xi_t(x) = 0$. Thus, we can also view $\xi_t$ as the subset $\set{x: \xi_t(x) = 1}$ of $V$ of the infected vertices at time $t$. Infected vertices transmit the infection to each of their neighbours with rate $\lambda$ and recover with rate $1$, yielding the transition rates given by \eqref{eq:transrate}. We write $(\xi^A_t)_{t\geq 0}$ for the contact process with initial condition $A\subset V$, i.e. $\xi_0^A = A$ and $(\xi^x_t)_{t\geq 0}$ if $A = \set{x}$.

\begin{figure}[!h]
\begin{center}
\begin{tikzpicture}[scale=0.4, every node/.style={scale=1}]
\draw (0,0) -- (14,0);
\foreach \x in {0,2,4,6,8,10,12,14}
	\draw (\x,0) -- (\x,12);
\foreach \x in {(0,4),(0,8.5),(2,3),(2,11),(4,7),(8,9),(10,2.8),(10,6),(12,5),(12,10.2),(14,3.7),(14,9.3)}
	\node at \x {$\times$};
\foreach \x in {(0,3.7),(2,9.5),(8,1.7),(10,5.6),(12,7.4)}
	\draw[->,thick] \x -- +(2,0);
\foreach \x in {(4,6.3),(6,5.5),(8,6.5),(10,8.4),(14,11.2)}
	\draw[->,thick] \x -- +(-2,0);
\foreach \x in {(0,5.1),(2,8.3),(4,6.6),(6,5),(8,3.5),(10,9.5),(12,4.7)}
\draw[cyan,very thick] (6,0) -- (6,5.5) -- (4,5.5) -- (4,6.3) -- (2,6.3) -- (2,9.5) -- (4,9.5) -- (4,12);
\end{tikzpicture}
\end{center}
\caption{Sketch of the graphical representation of the contact process on $\Z$. Arrows represent (potential) infection transmissions and crosses represent the recovery marks. The blue path is one potential infection path in this representation starting in the fourth vertex.}\label{fig:graphconst}
\end{figure}
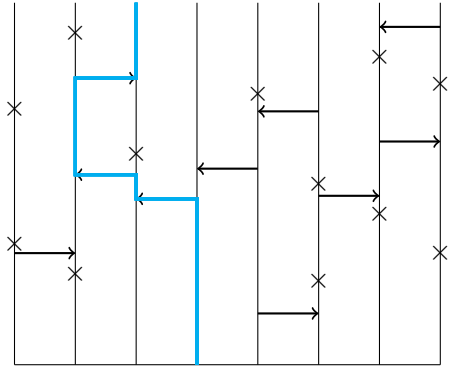

A very useful description of the contact process is its graphical representation given by a family of independent Poisson processes on $[0,\infty)$. Assign to each vertex $x \in V$ a Poisson process $N_x$ on $[0,\infty)$ of rate 1. For each edge in $G$ with endvertices $x$ and $y$, assign to each of the pairs $(x,y)$ and $(y,x)$ a Poisson process $N_{(x,y)}$, resp. $N_{(y,x)}$, on $[0,\infty)$ with rate $\lambda$.
We can think of every element $t\in N_x$ as a recovery mark at $x$ at time $t$, and every element $t\in N_{(x,y)}$ as a transmission arrow from $x$ to $y$ at time $t$. Hence, on $V\times [0,\infty)$ we assign a recovery mark at $(x,t)$ for all $t\in N_x$ and $x\in V$ and an arrow from $(x,t)$ to $(y,t)$ for all $t\in N_{(x,y)}$ and $x,y\in V$ which are connected by an edge. An \emph{infection path} in the graphical construction is a function $g:I\to V$ for some interval $I$ such that the process $(g(t),t)_{t\in I}$ on $V\times [0,\infty)$ which goes up in time never hits a recovery mark and only changes values in the first component by travelling along an arrow in its given direction. We write $(x,t) \to (y,s)$ if there exists an infection path $g:[t,s]\to V$ from $x$ to $y$, i.e. an infection path with $g(t) = x$ and $g(s) = y$. Then the contact process starting in $A$ can be derived from infection paths of the graphical construction by
\[
\xi_t^A(x) = \mathbf{1}\{A\times\{0\} \to (x,t)\},\quad t\geq 0, x \in V,
\]
see Figure \ref{fig:graphconst} for an example. This graphical representation allows to derive important properties of the contact process easily, such as its monotonicity in the initial configuration, its additivity and its \emph{self-duality} relation
\[
\P(\xi_t^A \cap B \neq \emptyset) = \P(\xi_t^B\cap A \neq \emptyset)\quad \text{for}\ A,B\subset V.
\]
For further properties of the contact process we refer to \cite{Lig1999}.

%In this work it will be primarly used in the form, where $B=\set{\x}$, which gives
%\[
%\P(\xi_t^A(\x) = 1) = \P(\xi_t^\x \cap A \neq \emptyset).
%\]
Throughout the following sections we denote by $\P$ the joint law of the contact process and the underlying geometric random graph and with a slight abuse of notation denote by $\P_{\x_1,\ldots,\x_n}$ and by $\P_{\Palm}$ the joint law of the contact process and the respective geometric random graph law under the conditions given in Section \ref{subsec:framework}.

\section{Non-extinction probability}\label{sec:non-ext_prob}
In this section we consider the probability $\Gamma(\lambda)$ that the contact process with parameter $\lambda$ on $\mathscr{G}_{\Palm}$ starting in the origin $(0,T_0)$ does not go extinct, i.e. 
\begin{equation}
\Gamma(\lambda) = \P_{\Palm}\big(\xi_t^{\Palm} \neq \emptyset\ \forall\, t\geq 0\big).
\end{equation}
Our main result describes the asymptotic behaviour of $\Gamma(\lambda)$ as $\lambda$ becomes small. We write $f(\lambda) \asymp g(\lambda)$, if there exist two positive constants $c,C>0$ such that $cf(\lambda) \leq g(\lambda) \leq Cf(\lambda)$ for $\lambda$ sufficiently small.

\begin{theorem}\label{thm:non-ext_prob}
Let $\mathscr{G}$ be a general geometric random graph which satisfies Assumptions \ref{ass:main} for $\gamma> \frac{\delta}{\delta+1}$. Then, as $\lambda\to 0$,
\begin{equation}
\Gamma(\lambda) \asymp \frac{\lambda^{2/\gamma-1}}{\log(1/\lambda)^{(1-\gamma)/\gamma}}. \label{eq:non-ext-proborder}
\end{equation}
\end{theorem}

As stated in the introduction for the ultrasmall regime the non-extinction probability is positive for any $\lambda>0$ and therefore the critical value when the contact process dies out is almost surely zero. To compare the result of Theorem \ref{thm:non-ext_prob} to the rates \eqref{eq:product-rates} given in \cite{MouVY2013}, resp. \cite{LinMSV2021}, for the contact process on the configuration model and on hyperbolic random graphs, note that $\gamma\in (\frac{\delta}{\delta+1},1)$ implies $\tau\in (2,3)$ since $\delta>1$. As $\tau = 1 + \frac{1}{\gamma}$ the rate given in \eqref{eq:non-ext-proborder} matches the case $\tau \in (\frac{5}{2},3]$ in \eqref{eq:product-rates}. 
To see the reason why for geometric random graphs satisfying Assumption \ref{ass:main} only this case appears, it is helpful to look at the survival strategies of the infection leading to the two cases of \eqref{eq:product-rates} for which $\tau\in (2,3]$. If $\tau\in (2,\frac{5}{2}]$, the graph is so well connected that an infected vertex with a high degree, i.e. a small mark, has with high probability at least one neighbour with an even smaller mark to which the vertex transmits the infection. Thus, when the origin infects a relatively powerful vertex, i.e. a vertex with a small mark, with high probability the infection passes directly to more and more powerful vertices and therefore survives. 
This way of \emph{direct spreading} does not work sufficiently well when $\tau\in (\frac{5}{2},3]$ as the graph is not connected well enough. In this case the survival strategy relies on the observation that, when a vertex with sufficiently high degree is infected, the infection survives so long in the neighbourhood of the vertex which forms a star, that it reaches with high probability another vertex with similarly high degree from which this kind of \emph{delayed spreading} can repeat, see \cite[Section 3]{MouVY2013}. 
For geometric random graphs satisfying Assumption \ref{ass:main} the strategy of direct spreading does not work, as for a vertex with mark $t$ the expected number of neighbours with mark smaller than $t$ does not increase when $t$ becomes small, unlike in the configuration model or hyperbolic random graphs. Instead, two vertices with small mark are usually not connected directly but via a \emph{connector}, a vertex with mark near one. This additional necessary step to transmit the infection to a vertex with smaller mark makes this strategy worse than the strategy of delayed spreading, which still works for the class of geometric random graphs, see Proposition \ref{prop:contact_lower_bound}, yielding that only the later one appears in Theorem \ref{thm:non-ext_prob}. The same behaviour also holds for dynamical non-spatial preferential attachment with slow update rate, as studied in \cite{JacLM2022}.

\begin{remark}
In recent work on a similar class of geometric random graphs the upper bound assumption could be relaxed by omitting the assumption of independent occurence of the edges given the Poisson point process, see \cite{GraGM2022}. This is not possible here, as the proof of the asymptotic upper bound given in Proposition \ref{prop:contact_upper_bound} requires not only the ability to control the occurence of self-avoiding paths in the graph, but also the occurence of stars, i.e. the neighbourhoods of vertices with high degree. An upper bound assumption on the existence of paths as in \cite[Assumption 1.1]{GraGM2022} does not yield any meaningful bound on the size of stars.
\end{remark}
%For the second result define $\mathscr{G}_n$ as the restriction of $\mathscr{G}$ to $[\big(\frac{n}{2}\big)^{-1/d},\big(\frac{n}{2}\big)^{1/d}]^d \times (0,1)$. The following result shows the convergence of the expected density of infected vertices in $\mathscr{G}_n$ to the non-extinction probability $\Gamma(\lambda)$ in $\mathscr{G}_n$.
%\begin{theorem}
%Let $\lambda>0$ and $\gamma\geq \frac{\delta}{\delta+1}$. Fix $(t_n)_{n\geq 1}$ such that $t_n \to \infty$. Then
%\[
%\lim_{n\to \infty} \E \big[\frac{\xi_{t_n}^{\mathscr{G}_n}}{n}\big] = \Gamma(\lambda).
%\]
%\todo{ToDo.}
%\end{theorem}

\subsection{Lower bound}\label{subsec:lowerbound}
We dedicate this section to proving a lower bound for $\Gamma(\lambda)$ when $\lambda$ is small. Namely, we will prove the following result.

\begin{prop}\label{prop:contact_lower_bound}
Let $\mathscr{G}$ be a general geometric random graph which satisfes Assumption \ref{ass:main} with $\gamma>\frac{\delta}{\delta+1}$. Then, there exists a constant $c>0$ such that, for $\lambda$ small, it holds
\begin{equation}
\Gamma(\lambda) > c \frac{\lambda^{2/\gamma-1}}{\log(1/\lambda)^{(1-\gamma)/\gamma}}.
\end{equation}
\end{prop}

For the proof of Proposition \ref{prop:contact_lower_bound} we exploit the following observation. Let us denote by a \emph{star} a connected graph where all but one vertex have degree one. Then, the contact process restricted to a subgraph isomorphic to a star survives for a constant time if the subgraph consists of at least order $\lambda^{-2}$ vertices and survives even long enough to infect other neighbouring stars if the subgraph consists of order $\log(1/\lambda)\lambda^{-2}$ vertices, see \cite{MouVY2013}. 
We denote by $\mathbb{L}_r$ the graph consisting of the half-line $\mathbb{N}_0$, where to each even vertex $m\in \mathbb{N}_0$, $r$ additional distinct neighbours are attached. Thus, $\mathbb{L}_r$ forms a half-line of stars consisting of $r+1$ vertices, where two consecutive stars are connected via a path of two edges. 
Throughout this section we denote the vertex $0\in \mathbb{N}_0$ as the origin of $\mathbb{L}_r$. Notice that if the size $r$ of the stars is of order $\log(1/\lambda)\lambda^{-2}$ and only the origin of $\mathbb{L}_r$ is infected, there exists a constant $p>0$ such that the survival probability of the contact process on $\mathbb{L}_r$ is at least $p$. This is direct consequence of \cite[Lemma 2.4]{LinMSV2021}, since the stars are sufficiently large so that  whether two stars are connected by a single edge or a path of bounded length makes no difference. 
Thus, as a first step we will show that such a half-line of stars exists in $\mathscr{G}_{\Palm}$. For $x\in \mathbb{R}^d$, denote by $H_x$ the plane through $x$ with normal vector $x$. Consequently, $H_x$ divides $\R^d$ in two subsets and we denote by $\R^d_{\geq x}$ the subset that does not contain zero. As discussed in Section \ref{subsec:framework} the expected degree of a given vertex in $\mathscr{G}$ with mark $t$ is of order $t^{-\gamma}$. We therefore call vertices with small mark powerful vertices. For $\beta>0$ and $r := \beta\log(1/\lambda)\lambda^{-2}$, let $T_{\sp} := r^{-1/\gamma}$ be the threshold such that vertices with smaller mark have an expected degree of order at least $\log(1/\lambda)\lambda^{-2}$. 

\begin{lemma}\label{lem:chainofstars}
Let $\x = (x,t) \in \mathbb{R}^d\times (0,1)$ with $t<T_{\sp}$. Then, given that $\x$ is a vertex of $\mathscr{G}_{\Palm}$, we have with high probability as $\lambda\to 0$ that there exists a subgraph of $\mathscr{G}_{\Palm}$ in $\R^d_{\geq x}\times (0,1)$ which is isomorphic to $\mathbb{L}_r$ such that the origin of $\mathbb{L}_r$ is identified with $\x$.
\end{lemma}

To prove the existence of such a subgraph, we decompose $\R^d_{\geq x}\times (0,1)$ into distinct parts, where areas with small marks represent potential midpoints of stars and areas with large mark represent either neighbours of the midpoints or vertices which are connected by an edge to two distinct midpoints. Choose $\theta>0$ such that $1<\theta<\gamma+\gamma/\delta$ and note that his is always possible since $\gamma>\frac{\delta}{\delta+1}$. Set 
\[
T_k = T_{\sp}^{\theta} e^{-k \theta}\ \text{and}\ R_k = \tfrac{1}{2}T_{\sp}^{-(\gamma+\gamma/\delta)/d}e^{k(\gamma+\gamma/\delta)/d},\quad \text{for}\ k\in \N
\] and $R_0 = 0$. Given the vertex $\x=(x,t)$ define, for $k\in \N$, the anuli $A_k := B(\x,R_k)\cap B(\x,R_{k-1})^c$ and the sets
\[
S_k := \tilde{A}_k \times [T_{k+1},T_k),\quad S^{(1)}_k := \tilde{A}_k \times [1/2,3/4)\ \text{and}\ S^{(2)}_k := \tilde{A}_k \times [3/4,1),
\]
where $\tilde{A}_k := A_k\cap \R^d_{\geq x}$, see Figure \ref{fig:anuli_struct}. Notice that all these sets are disjoint and therefore the point processes restricted to these sets are independent. For the proof of Proposition \ref{lem:chainofstars} it will be helpful to interpret 
\begin{itemize}
\item the vertices in $S_k$ as the potential midpoints of the $k$-th star of the subgraph,
\item the vertices in $S^{(1)}_{k}$ as the potential neighbours of the midpoints forming a sufficiently large star
\item the vertices in $S^{(2)}_k$ as the potential \emph{connectors} between consecutive midpoints, i.e. vertices which are connected to both midpoints.
\end{itemize}

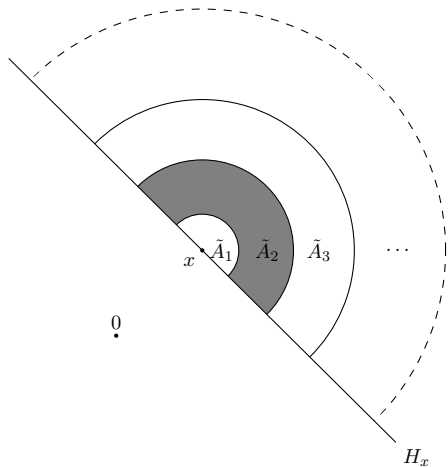
\begin{figure}[]
\begin{center}
\begin{tikzpicture}[scale=0.4, every node/.style={scale=0.7}]
\begin{scope}
    \clip[rotate around={45:(0,0)}] (0,-8) rectangle (8,8);
    \draw[fill=gray] (0,0) circle(3);
    \draw[fill=white] (0,0) circle(1.2);
    \draw (0,0) circle(5);
    \draw[dashed] (0,0) circle(8);
\end{scope}
    \fill (0,0) node[below left] {$x$} circle[radius=2pt]; 
    \draw[rotate around={45:(0,0)}] (0,-9) node[below right] {$H_x$} -- (0,9);
    \fill[rotate around={45:(0,0)}] (-4,0) node[above] {$0$} circle[radius=2pt]; 
    \node[right] at (0,0) {$\tilde{A}_1$};
    \node[right] at (1.5,0) {$\tilde{A}_2$};
    \node[right] at (3.2,0) {$\tilde{A}_3$};
    \node[right] at (5.8,0) {$\cdots$};
\end{tikzpicture}
\caption{The annuli $(A_k)_{k\in \N}$ centered around $x$ and truncated with respect to $H_x$, where the area $\tilde{A}_2$ is shaded grey.}\label{fig:anuli_struct}
\end{center}
\end{figure}

%\begin{figure}[h]
%\begin{center}
%\begin{tikzpicture}[scale=0.4, every node/.style={scale=0.7}]
%\draw (-1,7) node[left] {$1$} -- (-1,3) node[left] {$\frac{1}{2}$} --  (-1, -1) node[left] {$0$};
%\draw (0,-1) -- (9,-1) node[below] {$\mathbb{R}$}--  (18,-1);
%\draw (1,3) rectangle ++(2,2); \node at (2,4) {$S^{(1)}_{k-1}$};
%\draw (1,5) rectangle ++(2,2); \node at (2,6) {$S^{(2)}_{k-1}$};
%\draw (3,3) rectangle ++(4,2); \node at (5,4) {$S^{(1)}_k$};
%\draw (3,5) rectangle ++(4,2); \node at (5,6) {$S^{(2)}_k$};
%\draw (7,3) rectangle ++(8,2); \node at (11,4) {$S^{(1)}_{k+1}$};
%\draw (7,5) rectangle ++(8,2); \node at (11,6) {$S^{(2)}_{k+1}$};
%
%\draw (1,1) rectangle ++(2,1.5); \node at (2,1.75) {$S_{k-1}$};
%\draw (3,0) rectangle ++(4,1); \node at (5,0.5) {$S_{k}$};
%\draw (7,-0.75) rectangle ++(8,0.75); \node at (11,-0.375) {$S_{k+1}$};
%\end{tikzpicture}
%\caption{Sketch of the structure of the boxes for some $k\in \mathbb{N}$ and its neighbours in dimension one.}
%\end{center}
%\end{figure}

Before we prove Lemma \ref{lem:chainofstars}, we state the following lemma as a corollary of \cite{JacM2017}, which follows with similar calculations as done in \cite{GraGM2022}. With this lemma we can ensure that two vertices from consecutive areas $S_k$ and $S_{k+1}$ are connected via a connector with high probability.

\begin{lemma}\label{lem:twoconn}
Given two vertices $\x_k\in S_k$ and $\x_{k+1}\in S_{k+1}$, the number of vertices in $S_k^{(2)}$ which form an edge to $\x_k$ and $\x_{k+1}$ is Poisson-distributed with parameter larger than 
\begin{equation}
C t_k^{-\gamma} \big(1\wedge t_{k+1}^{-\gamma\delta}\big(\abs{x_{k+1}-x_k} + t_{k}^{-\gamma/d}\big)^{-d\delta}\big), \label{eq:twoconn}
\end{equation}
where $C>0$ is a constant not depending on $k$.
\end{lemma}
\begin{proof}
We consider the vertices $\z = (z,u) \in S_k^{(2)}$ with $\abs{x_k-z}^d < t_k^{-\gamma}$. Note that the volume of $B\big(x_k,t_k^{-\gamma/d}\big)\cap \tilde{A}_k$ is a positive proportion $\rho>\frac{1}{2^{d+1}}$ of the ball volume, as $t_k^{-\gamma/d} < \frac{1}{2}(R_k-R_{k-1})$ for sufficiently small $\lambda$. Then, the number of those vertices which form an edge to $\x_k$ and $\x_{k+1}$ is Poisson-distributed with parameter bounded from below by
\begin{align*}
\int_{\frac{3}{4}}^{1} \int_{B(x_k,t_k^{-\gamma/d})\cap \tilde{A}_k} \mathd \z &\alpha^2(1\wedge \kappa_1)(1\wedge \kappa_1 t_{k+1}^{-\gamma\delta} \abs{z-x_{k+1}}^{-d\delta})\\
\geq &\frac{V_d \rho\alpha^2(1\wedge \kappa_1)}{4}t_k^{-\gamma}\big(1\wedge \kappa_1 t_{k+1}^{-\gamma\delta}\big(\abs{x_{k+1}-x_k} + t_{k}^{-\gamma/d}\big)^{-d\delta}\big),
\end{align*}
where $V_d$ is the volume of the $d$-dimensional unit ball and $\alpha$ comes from Assumption \ref{ass:main}. Thus, there exist a constant $C>0$ sufficiently small such that \eqref{eq:twoconn} holds.
\end{proof}

\begin{proof}[Proof of Lemma \ref{lem:chainofstars}]
Note that, for $k\in \N$, the number of vertices in $S_k$ is Poisson-distributed with mean larger than
\[
V_d(T_{k+1}-T_k)(R_{k+1}^d-R_k^d) \geq c_1 T_{\sp}^{-(\gamma+\gamma/\delta-\theta)}e^{k(\gamma+\gamma/\delta-\theta)}
\]
for some $c_1>0$, which does not depend on $k$. Thus, $S_k$ is non-empty with probability larger than $1-\exp(-c_1 T_{\sp}^{-(\gamma+\gamma/\delta-\theta)}e^{k(\gamma+\gamma/\delta-\theta)})$. As the boxes $S_1,S_2,\ldots$ are disjoint, the numbers of vertices in each box are independent from each other. Thus, it holds
\[
\P(S_k\ \text{is non-empty}\ \forall\, k\in \mathbb{N})\geq \prod_{k=1}^\infty \bigg(1-\exp(-c_1 T_{\sp}^{-(\gamma+\gamma/\delta-\theta)}e^{k(\gamma+\gamma/\delta-\theta)})\bigg).
\]
For $k\geq 2$, on the event that $S_k$ is non-empty we denote by $\x_k$ the vertex in $S_k$ with the smallest mark. For $k=1$ we set $\x_1:= \x$ and always treat $S_1$ as non-empty. To keep notation cleaner we treat without loss of generality the mark of $\x_1$ as smaller than $T_1$ (this affects only the estimate in \eqref{eq:lowbound_connprob} below where $\gamma+\gamma\delta$ becomes $\gamma\delta$ which does not change the rest of the argument).

Given $S_k$ and $S_{k+1}$ are non-empty, the Euclidean distance of the corresponding vertices $\x_k$ and $\x_{k+1}$ is at most $2R_{k+1}$ and both vertices have a mark smaller than $T_k$. Thus, by Lemma \ref{lem:twoconn} there exists $c_2>0$ such that, for all $k\in \N$, the probability that $\x_k$ and $\x_{k+1}$ are connected via one connector in $S_k^{(2)}$ is larger than 
\begin{equation}
1-\exp(c_2T_{\sp}^{-(\gamma+\gamma\delta)(\theta-1)}e^{k(\gamma+\gamma\delta)(\theta-1)}). \label{eq:lowbound_connprob}
\end{equation}
%Combining these two observations yields that there exists $\varepsilon>0$ such that, given $\x_k \in S_k$, the probability that there exists a vertex $\x_{k+1}$ in $S_{k+1}$ which connected to $\x_k$ via a connector in $S_k^{(2)}$ is larger than $1-\exp(T_*^{-\varepsilon}e^{k\varepsilon})$.

Given $S_k$ is non-empty, we now turn our attention to the number of neighbours of $\x_k$ in $S_k^{(1)}$. As in the proof of Lemma \ref{lem:twoconn} we only consider the vertices $\z\in S_k^{(1)}$ with $\abs{x_k-z}^d<T_k^{-\gamma}$. Note that the volume of $B(x_k,T_k^{-\gamma/d})\cap A_k$ is again a positive proportion $\rho>\frac{1}{2^{d+1}}$ of the ball volume itself for $\lambda$ small enough. These vertices are connected to $\x_k$ with probability bounded from below by $\alpha(1\wedge \kappa_1)$. Thus, there exists $c>0$ such that the number of neighbours of $\x_k$ in $S_{k}^{(1)}$ with $\abs{x_k-z}^d<T_k^{-\gamma}$ is Poisson-distributed with mean larger than $cT_k^{-\gamma} \geq cre^{k\gamma}$. Thus, by a Chernoff-bound there exists $c_3>0$ such that, given $S_k$ is non-empty, the probability that the number of neighbours of $\x_k$ is larger than $r$  is at least $1-\exp(-cre^{k\gamma})$, for all $k\in \N$.

We choose $0<\varepsilon<\big(\gamma+\gamma/\delta-\theta\big) \wedge \big((\gamma+\gamma\delta)(\theta-1)\big)$. Given all sets $S_k$ are non-empty, the two previously discussed events only depend on disjoint subsets of the Poisson-process. Thus, the probability that, for all $k\geq 2$,
\begin{itemize}
\item $S_k$ is non-empty,
\item $\x_k$, the vertex with smallest mark in $S_k$, has at least $r$ neighbours in $S_k^{(1)}$ and
\item $\x_{k-1}$ and $\x_{k}$ are connected via a connector in $S_{k-1}^{(2)}$
\end{itemize}
is larger than
\[
\prod_{k=1}^\infty \big(1-\exp(-T_{\sp}^{-\varepsilon}e^{k\varepsilon})\big)\big(1-\exp(-cre^{k\gamma})\big)
\]
which tends to one as $\lambda\to 0$ since $r$ and $T_{\sp}^{-\varepsilon}$ tend to infinity in this case.
\end{proof}

With Lemma \ref{lem:chainofstars} in hand we are now ready to complete the proof of Proposition \ref{prop:contact_lower_bound}. Starting at the origin we explore the Poisson point process by expanding a sphere centered at the origin $(0,T_0)$ until we find the nearest neighbour $\x$ of $(0,T_0)$ with mark smaller than $T_{\sp}$. As the number of neighbours of $(0,T_0)$ with mark smaller than $T_{\sp}$ dominates a Poisson-distributed random variable with parameter of order $T_{\sp}^{1-\gamma}$, the probability that we find such a neighbour $\x$ and there is a transmission from $(0,T_0)$ to $\x$ before $(0,T_0)$ recovers is larger than
\begin{equation}
 \frac{c\lambda}{1+\lambda}T_{\sp}^{1-\gamma}\geq \frac{c\lambda^{2/\gamma-1}}{\log(1/\lambda)^{(1-\gamma)/\gamma}} \label{eq:nearneighbour}
\end{equation}
for some $c>0$, where $\frac{\lambda}{\lambda+1}$ occurs as the probability that, given $\x$ is a neighbour of the origin $(0,T_0)$, $\x$ got infected by the origin before it recovers. Given the nearest neighbour $\x=(x,t)$ with mark smaller than $T_{\sp}$, by Lemma \ref{lem:chainofstars}, there exists with probability larger than $\frac{1}{2}$ as $\lambda$ is small a subgraph in the yet unexplored area $\mathbb{R}_{\geq x}^d$ which is isomorphic to $\mathbb{L}_r$ with origin in $\x$. Conditioned on this subgraph being present and the origin of it being infected, the infection survives with a probability bounded away from zero, uniformly in $\lambda$, by \cite[Lemma 2.4]{LinMSV2021}. Hence, $\Gamma(\lambda)$ is up to a constant larger than the probability bound given in \eqref{eq:nearneighbour} which completes the proof of Proposition \ref{prop:contact_lower_bound}.

\subsection{Upper bound}
\begin{prop}\label{prop:contact_upper_bound}
Let $\mathscr{G}$ be a general geometric random graph which satisfies Assumption \ref{ass:main} for $\gamma>\frac{1}{2}$ and $\delta>1$. Then, there exists a constant $C>0$ such that
\begin{equation}
\Gamma(\lambda) < C \frac{\lambda^{2/\gamma-1}}{\log(1/\lambda)^{(1-\gamma)/\gamma}}.
\end{equation}
\end{prop}

To prove this we generalize and extend the arguments from \cite{LinMSV2021} for the geometric random graphs characterized by the framework given in Section \ref{subsec:framework}. Let the function $\rho:[0,\infty)\to [0,1]$ be defined by $\rho(x) := 1\wedge x^{-\delta}$ and denote $I_\rho
:= \int_{\R^d}\mathd x \rho(\kappa_2\abs{x}^d)<\infty$, where $\delta>1$ and $\kappa_2>0$ are given in Assumption \ref{ass:main}.
Then, by Assumption \ref{ass:main} there exists constants $c,C>0$ only depending on the parameter $\gamma$ and $I_\rho$ such that, for any vertex $\x = (x,t)\in \R^d\times (0,1)$, it holds $ct^{-\gamma} \leq 
\E_\x\deg{\x} \leq Ct^{-\gamma}$. With that in mind, let $T(n) = n^{-1/\gamma}$ be the mark of a vertex with expected degree of order $n$.

Throughout the proof we classify the vertices by their expected degree into different groups. Let $n_{\st} = \lambda^{-2}$ and, for some constant $\theta>0$ to be specified later, $n_{\sp} = \frac{\theta}{\lambda^2}\log\big(\frac{1}{\lambda}\big)$. Vertices with degree larger than $n_{\st}$ are the midpoints of stars on which the infection restricted to the star can survive for a constant time with a probability bounded away from zero, without necessarily surviving long enough such that it can reach other stars nearby. As we have seen in the proof of Proposition \ref{prop:contact_lower_bound}, this happens for stars with more than $n_{\sp}$ vertices. For $\sigma>0$, set $n_\sigma=\lambda^{-2+\sigma}$. Vertices with degree smaller than $n_\sigma$ are centers of stars which are not sufficiently large, in the sense that the infection does not propagate through such stars and dies out when the graph is restricted to vertices with such small degree; see Lemma \ref{lem:weak_paths_not_inf}.
We denote by
\[
T_{\st} := T(n_{\st}),\quad T_\sigma := T(n_\sigma),\quad T_{\circledast} := T(n_{\sp})
\]
the according mark of the vertices whose expected degree is of the corresponding order.

\begin{proof}[Proof of Proposition \ref{prop:contact_upper_bound}]
We consider now the contact process $(\xi^{\Palm}_t)_{t\geq 0}$ on $\mathscr{G}_{\Palm}$ which starts from the origin $(0,T_0)$. For a vertex $\x$, on the event that $\x$ and $(0,T_0)$ are connected we denote by $I_\x$ the event that $\x$ got infected from the origin $(0,T_0)$ before $(0,T_0)$ recovers. On $I_\x$ we denote by $\tau_\x$ the time when $\x$ got infected by the origin and by $(\eta_t^\x)_{t\geq \tau_\x}$ the contact process started at time $\tau_\x$ with a single infection at $\x$ determined by the same graphical construction as the original process $(\xi^{\Palm}_t)_{t\geq 0}$.

For $\sigma>0$, denote by $E_\sigma$ the event that each infection path $g$ starting in the origin which jumps first to a vertex with mark larger than $T_\sigma$ is finite and never reaches a vertex with mark smaller than $T_\sigma$. We will see later that when the mark of the origin is itself larger than $T_\sigma$, the probability that $E_\sigma$ occurs goes much quicker to one, when $\lambda$ goes to zero, than the rate given in \eqref{eq:non-ext-proborder}. In fact, we show in Lemma \ref{lem:weak_paths_not_inf} that for $\sigma>0$ sufficiently small, it holds
\begin{equation}
\P_{\Palm}(E_\sigma^c\cap \{T_0\geq T_\sigma\}) \leq \lambda^{2/\gamma-1+\varepsilon}.\label{eq:weak_paths_not_inf}
\end{equation}

Let $\sigma_0>0$ such that it holds $\sigma_0>\sigma$ and let $T_{\sigma_0} = T(n_{\sigma_0})$ be the associated boundary of the mark of vertices with expected degree of order $\lambda^{-(2-\sigma_0)}$. Whereas $T_\sigma$ and $T_{\sp}$ will be used to distinguish the neighbours of the origin by their marks, for the survival of the contact process we consider whether $T_0< T_{\sigma_0}$ or not. Then, we have
\begin{align}
\begin{split}
\1_{\{\xi_t^{\Palm} \neq \emptyset\ \forall\, t\geq 0\}} &\leq \1_{\{T_0< T_{\sigma_0}\}} + \1_{E_\sigma^c\cap \{T_0\geq T_\sigma\}}\\
&+ \sum_{\substack{\x\in \X\\ t\leq T_{\sp}}} \1_{\{(0,T_0)\sim \x\}}\1_{I_\x}\1_{\{T_0\geq T_{\sigma_0}\}}\\
&+ \!\!\!\!\!\sum_{\substack{\x\in \X\\ T_{\sp}<t< T_{\sigma}}}\!\!\! \1_{\{(0,T_0)\sim \x\}}\1_{I_\x}\1_{\{T_0\geq T_{\sigma_0}\}}\1_{\{(\eta_t^{\x})_{t\geq \tau_\x}\ \text{survives}\}}
\end{split}\label{eq:sum_bound}
\end{align}
In fact, if the right-hand side is zero it holds that
\begin{itemize}
\item every infection path starting in the origin which jumps in its first step to a vertex with mark larger than $T_\sigma$ is finite and never visits a vertex with mark smaller than $T_\sigma$,
\item there exists no vertex with mark smaller than $T_{\sp}$ which is a neighbour of the origin and got infected by it,
\item there exists no vertex with mark inbetween $T_\sigma$ and $T_{\sp}$, which is a neighbour of the origin and got infected by it and the infection emerging from this vertex survives.
\end{itemize}
As these three points imply that the infection $(\xi_t^{\Palm})_{t\geq 0}$ does not survive, the left-hand side is then also zero, from which the stated inequality follows. Note that this bound is especially not sharp in the second summand of the right-hand side, as we allow $T_0$ to take a larger range of values than needed. Taking expectation on both sides of \eqref{eq:sum_bound} yields an upper bound for the probability of interest and it is therefore sufficient to establish upper bounds for the expectation of each of the four summands of the right-hand side.

First, by the definition of $T_{\sigma_0}$ we have that 
\[
\P(T_0<T_{\sigma_0}) = \lambda^{(2-\sigma_0)/\gamma} < \lambda^{2/\gamma-1/2}
\]
for $\sigma_0$ and $\lambda>0$ small enough and a bound for the second summand is given by \eqref{eq:weak_paths_not_inf}, which will be proved in Lemma \ref{lem:weak_paths_not_inf}.

As mentioned beforehand, the third term on the right-hand side of \eqref{eq:sum_bound} counts the number of vertices with mark smaller than $T_{\sp}$ which got infected by the origin $(0,T_0)$. We have seen in Section \ref{subsec:lowerbound} that the contact process starting in such a vertex would ensure a spreading over a chain of infinitely many other stars with equally many neighbours. Thus, a sharp upper bound for the expected number of such vertices is crucial. By Mecke's equation, Assumption \ref{ass:main} and since $T_0$ is independent of $\X$ it holds that
\begin{align*}
\E_{(0,T_0)}\bigg[\sum_{\substack{\x\in \X\\ t\leq T_{\sp}}} \1_{\{(0,T_0)\sim \x\}}\1_{I_\x}\1_{\{T_0\geq T_{\sigma_0}\}}\bigg] &= \int_{T_{\sigma_0}}^1 \mathd t_0 \int_0^{T_{\sp}} \mathd t \int_{\R^d} \mathd x\, \E_{(0,t_0),\x}[\1_{\{(0,t_0)\sim \x\}}\1_{I_\x}]\\
&\leq \frac{\lambda}{\lambda+1}\int_{T_{\sigma_0}}^1 \mathd t_0 \int_0^{T_{\sp}} \mathd t \int_{\R^d} \mathd x \,\rho(\kappa_2^{-1/\delta}t^\gamma t_0^{1-\gamma}\abs{x}^d)\\
&\leq \lambda\frac{I_\rho}{(1-\gamma)\gamma} T_{\sp}^{1-\gamma} \leq \frac{I_\rho}{(1-\gamma)\gamma}\frac{\lambda^{2/\gamma-1}}{\log(1/\lambda)^{(1-\gamma)/\gamma}}.
\end{align*}
In fact, this is the dominant term of the right-hand side of \eqref{eq:sum_bound} which contributes to the stated upper bound.

For the last summand of \eqref{eq:sum_bound} we have by Mecke's equation, Assumption \ref{ass:main} and since $T_0$ is independent of $\X$ that
\begin{align*}
&\E_{(0,T_0)}\bigg[\sum_{\substack{\x\in \X\\ T_{\sp}<t< T_{\sigma}}}\,\, \1_{\{(0,T_0)\sim \x\}}\1_{I_\x}\1_{\{T_0\geq T_{\sigma_0}\}}\1_{\{(\eta_t^{\x})_{t\geq \tau_\x}\ \text{survives}\}}\bigg]\\
&= \int_{T_{\sigma_0}}^1 \mathd t_0 \int_{T_{\sp}}^{T_\sigma} \mathd t \int_{\R^d} \mathd x \E_{(0,t_0),\x}[\1_{\{(0,t_0)\sim \x\}}\1_{I_\x}\1_{\{(\eta_t^{\x})_{t\geq \tau_\x}\ \text{survives}\}}]\\
&\leq \frac{\lambda}{\lambda+1} \int_{T_{\sigma_0}}^1 \mathd t_0 \int_{T_{\sp}}^{T_\sigma} \mathd t \int_{\R^d} \mathd x \E_{(0,t_0),\x}[\1_{\{(\eta_t^{\x})_{t\geq \tau_\x}\ \text{survives}\}}\, \vert\, (0,t_0)\sim \x, I_\x] \rho(\kappa_2^{-1/\delta}t^\gamma t_0^{1-\gamma}\abs{y}^d).
\end{align*}

In Lemma \ref{lem:non_survival_mid_vertices} we will show that there exist $\varepsilon>0$ such that, for $\x=(x,t),\ \y=(y,s) \in \R^d\times(0,1)$ with $t>T_{\sp}$ and $s>T_{\sigma_0}$, it holds
\begin{equation}
\P_{\x,\y}(\xi_t^\x \neq \emptyset\ \forall\ t\geq 0\,\vert\, \x\sim \y)<\lambda^\varepsilon, \label{eq:non_survival_mid_vertices}
\end{equation}
where $(\xi_t^\x)_{t\geq 0}$ is the contact process of rate $\lambda$ which starts in $\x$ and only in $\x$. By this inequality and the strong Markov-property of the contact process we have that
\begin{align*}
&\E_{(0,T_0)}\bigg[\sum_{\substack{\x\in \X\\ T_{\sp}<t< T_{\sigma}}}\,\, \1_{\{(0,T_0)\sim \x\}}\1_{I_\x}\1_{\{T_0\geq T_{\sigma_0}\}}\1_{\{(\eta_t^{\x})_{t\geq \tau_\x}\ \text{survives}\}}\bigg]\\
&\leq \lambda^{1+\varepsilon} \int_{T_{\sigma_0}}^1 \mathd t_0 \int_{T_{\sp}}^{T_\sigma} \mathd t \int_{\R^d} \mathd x\rho(\kappa_2^{-1/\delta}t^\gamma t_0^{1-\gamma}\abs{x}^d)\\
&\leq \frac{I_\rho}{(1-\gamma)\gamma} \lambda^{2/\gamma-1+\varepsilon-\sigma(1/\gamma-1)} < \frac{I_\rho}{(1-\gamma)\gamma}\lambda^{2/\gamma-1+\varepsilon/2}
\end{align*}
for $\sigma>0$ and $\lambda>0$ sufficiently small which completes the proof.
\end{proof}

We now proceed to establish the probability bounds \eqref{eq:weak_paths_not_inf} and \eqref{eq:non_survival_mid_vertices} of the previous proof which have been left out. To show these bounds we need to have control over the occurrence of infection paths which corresponds to both events, i.e. infection paths which jump from the origin to a vertex with mark larger than $T_\sigma$ and infection paths starting in a vertex with mark between $T_\sigma$ and $T_{\sp}$. For these bounds we will rely on the arguments used in \cite[Section 5]{LinMSV2021}, for which we will give a short overview. 

We consider now a graph $G = (V,E)$ with root $0$ and let $P$ be the set of all finite and infinite paths of vertices in the graph $G$. Instead of looking at infection paths themselves, we look at the paths of vertices which result from infection paths by capturing the visited vertices. Precisely, for an infection path $g:I\to V$ we define its \emph{ordered trace} $p_g\in P$ as the path of vertices in $G$ given by the vertices visited by $g$ in the same order. The following result by \cite{LinMSV2021} shows the usefulness of this definition, as to control the occurence of a given class of infection paths, it is sufficient to control the number of ordered traces corresponding to this class.

\begin{lemma}(\cite[Lemma 5.1]{LinMSV2021})\label{lem:cite_1}
Let $\lambda<\frac{1}{2}$. Given $p\in P$, the probability that there exists $t\geq 0$ and an infection path $g:[0,t]\to V$ with $p$ as its ordered trace is at most $(2\lambda)^{\abs{p}}$.
\end{lemma}

We define the following subsets of $P$. Let $A\subset V$ such that the root $0$ is not in $A$ and define
\begin{itemize}
\item for $n\geq 1$, $Q_A^n$ as the set of paths in $G$ of length $n$ starting in $0$, where the first $n$ vertices are distinct and not in $A$ but the last vertex is in $A$,
\item for $n\geq 3$, $R_{A}^n$ as the set of paths in $G$ of length $n$ starting in $0$, where the first $n$ vertices are distinct and not in $A$ but the last vertex is equal to a previous one.
\end{itemize}
We denote $Q_A := \bigcup_{k\geq 1} Q_A^k$ and $R_A := \bigcup_{k\geq 3} R_A^k$. \bigskip \\
Typically $A$ is a set of vertices with small mark, for example smaller than $T_{\st}$ or $T_{\sp}$. These vertices are powerful as they have many neighbours and therefore ensure the survival of the infection for longer time. The set $Q_A^n$ then describes all ordered traces which get to such a powerful vertex in their $n$-th step. The motivation for the second set $R_A^n$ lies in the following result from \cite{LinMSV2021}. Given a vertex $x \in V$, the contact process $(\xi_t)_{t\geq 0}$ starting in $0$ is \emph{thin on $x$} if there is no infection path $g:[0,t]\to V$ for some $t\geq 0$ with $g(0) = 0$ where $x$ appears more than once in the ordered trace $p_g$ of $g$. The contact process starting in $0$ is thin on a set $V'\subset V$ if it is thin on every vertex of $V'$.
\begin{lemma}(\cite[Lemma 5.4]{LinMSV2021})\label{lem:cite_2}
If $V_0\subset V$ is finite, then on the event that  $(\xi_t)_{t\geq 0}$ is thin on $(V_0)^c$, the contact process almost surely dies out, that is, almost surely there is $t\geq 0$ such that $\xi_t = \emptyset$.
\end{lemma}
Ordered traces which are not in $Q_A \cup R_A$ do not connect to a vertex with high degree and visit each vertex at most once. Thus, by the previous lemma infection paths which have such ordered traces typically do not contribute to the survival of the contact process. A consequence is the following key result by \cite{LinMSV2021}
\begin{lemma}(\cite[Lemma 5.5]{LinMSV2021})\label{lem:cite_3}
There exists $c>0$ such that, for any $\lambda<\frac{1}{2}$, the following holds. Let $G$ be a graph with root $0$, $A\subset V$ and let $(\xi_t)_{t\geq 0}$ be the contact process on $G$ starting in $0$. Then,
\[
\P(\xi_t\neq \emptyset\ \forall t\geq 0) \leq \frac{\exp\big(c\lambda^2\deg(\0)\big)}{T}+T \sum_{p\in Q_A\cup R_A}(2\lambda)^{\abs{p}}\quad \text{for all}\ T>0.
\]
\end{lemma}

We consider again geometric random graphs given by our framework in Section \ref{subsec:framework} and proceed to establish bounds for the number of ordered traces in $Q_A$ and $R_A$ on $\mathscr{G}$ for a suitable choice of $A$. As $A$ is a set of vertices typically defined in terms of their marks, the ordered traces in $Q_A$ and $R_A$ have restrictions on the marks of the vertices but not on their location. Thus, to control the number of such ordered traces we rely on the following definition as done in \cite{GraGM2022}. Let $\ell\in (0,1)$ be the truncation value and $\kappa\geq I_\rho$. We define, for $n\in \mathbb{N}$ and $t_0 \in (0,1)$, 
\begin{align}
\nu_{\ell,n}^{t_0}(s) :=   \int\limits_{\ell}^1 \mathd t_1 \cdots \!\!\! & \int\limits_{\ell}^1 \mathd t_{n-1}  \notag\\ & \!\!\!\!\!\!\!\!\!\!\!\! \kappa (t_{n-1}\wedge s)^{-\gamma}(t_{n-1}\vee s)^{\gamma-1}\prod_{k=1}^{n-1} \kappa (t_{k-1}\wedge t_k)^{-\gamma}(t_{k-1}\vee t_k)^{\gamma-1}\quad \text{for}\ s\in (0,1) \label{eq:non-spatial_mu_def}
\end{align}
and set $\nu_{\ell,0}^{t_0}(s) = \delta_0(t_0-s)$. Note that as defined, $\nu_{\ell,n}^\x(s)$ can be written recursively as
\begin{equation}
\nu_{\ell,n}^{t_0}(s) = \int\limits_{\ell}^1 \mathd u \, \nu_{\ell,n-1}^{t_0}(u)\kappa (u\wedge s)^{-\gamma}(u\vee s)^{\gamma-1}. \label{eq:non-spatial_mu_rec}
\end{equation}
This allows us to establish an upper bound for $\nu_{\ell,n}^\x(s)$ analogous to the non-spatial case in \cite[Lemma 1]{DerMM2012}. This is a corollary of \cite[Lemma 2.7]{GraGM2022}.
\begin{lemma}\label{lem:nu_bound}
Let $\ell\in (0,1)$ and $\nu_{\ell,n}^{t_0}(s)$ be as defined in \eqref{eq:non-spatial_mu_def}, where $\x = (x_0,t_0)$ and $s\in (0,1)$. Then, there exists a constant $c>0$, independent of $\ell$, such that for all $n\geq 2$,  
\begin{equation}
\nu_{\ell,n}^{t_0}(s) \leq \alpha_n s^{-\gamma} + 1\{s\geq \ell\}\beta_n s^{\gamma-1}, \label{eq:nu_bound}
\end{equation}
where 
\begin{align}
\begin{split}\alpha_{n+1} &= c \big(\alpha_n \log\big(\tfrac{1}{\ell}\big) + \beta_n\big)\\
\beta_{n+1} &= c \big(\alpha_n \ell^{1-2\gamma} + \beta_n\log\big(\tfrac{1}{\ell}\big)\big) \label{eq:alp_bet_def}
\end{split}
\end{align}
and $\alpha_1 = \kappa t_0^{\gamma-1}$, $\beta_1 = \kappa t_0^{-\gamma}$.
\end{lemma}
\begin{proof}
This follows analogously to the proof of \cite[Lemma 2.7]{GraGM2022} by choosing a fixed truncation value $\ell$ instead of an arbitrary truncation sequence.
\end{proof}
The following result gives an explicit upper bound for the sequence $(\alpha_n)_{n\in \mathbb{N}}$.

\begin{lemma}
Let $\ell>0$ be sufficiently small and $(\alpha_n)_{n\in \mathbb{N}}$, $(\beta_n)_{n\in \mathbb{N}}$ the sequences defined by \eqref{eq:alp_bet_def}. Then, it holds\begin{equation}
\alpha_n \leq (2c)^{n-2}c^2\big(\ell^{\frac{1}{2}-\gamma}t_0^{\gamma-1}+t_0^{-\gamma}\big)\ell^{(1-2\gamma)(\frac{n}{2}-1)}\quad \text{for}\ n\geq 2.\label{eq:alp_bound}
\end{equation}
\end{lemma}
\begin{proof} Let $\ell$ be small enough such that $c^{-2}<\log(\frac{1}{\ell})^2<\frac{1}{3}\ell^{1-2\gamma}$, where $c>0$ is the constant given in Lemma \ref{lem:nu_bound}. Then, it is easy to see that \eqref{eq:alp_bound} holds for $n=2$.
We show that it holds that
\begin{equation}
\alpha_{n+2} \leq (2c)^2\ell^{1-2\gamma}\alpha_n\ \text{for}\ n\geq 2. \label{eq:alp_bound_proof1}
\end{equation}
By the definition $(\alpha_n)_{n\in \N}$ and $(\beta_n)_{n\in \N}$ it holds that
\begin{equation}
\alpha_{n+2} = c^2\big(\log(\tfrac{1}{\ell})^2 + \ell^{1-2\gamma}\big)\alpha_n + 2c^2\log(\tfrac{1}{\ell})\beta_n\quad \text{for}\ n\geq 1.\label{eq:alp_bound_proof2}
\end{equation}
As \eqref{eq:alp_bet_def} also implies that $\alpha_n\vee \beta_n\leq \alpha_{n+1}$ for $n\in \N$ it is easy to see that it also holds that
\begin{equation}
\log(\tfrac{1}{\ell})\beta_n\leq \ell^{1-2\gamma}\alpha_n + \log(\tfrac{1}{\ell})^2\alpha_n\ \text{for}\ n\geq 2.\label{eq:alp_bound_proof3}
\end{equation}
Combining both inequalities \eqref{eq:alp_bound_proof2} and \eqref{eq:alp_bound_proof3} yields \eqref{eq:alp_bound_proof1}. Note that \eqref{eq:alp_bound_proof2} also implies that \eqref{eq:alp_bound} holds for $n=3$. Thus, \eqref{eq:alp_bound} follows directly by induction with \eqref{eq:alp_bound_proof1}.
\end{proof}

Recall that 
\[
T_{\sp} = \frac{\lambda^{2/\gamma}}{\theta^{1/\gamma}}\log(1/\lambda)^{-1/\gamma},\ T_{\st} = \lambda^{2/\gamma}\ \text{and}\ T_{\sigma} = \lambda^{(2-\sigma)/\gamma}.
\]
For $\x,\y\in \R^d\times (0,1)$ with the mark of $\x$ satisfying $t>T_{\sp}$ and the mark of $\y$ satisfying $s>T_{\sigma_0}$, we consider now the geometric random graph $\mathscr{G}$ under the law $\P_{\x,\y}(\cdot\,\vert\, \x\sim\y)$. 
On the event that $\x$ and $\y$ are vertices in $\mathscr{G}$, we declare $\x$ as the root of the graph and set 
\begin{equation}
A := \{\z=(z,u)\in \X:u\leq T_{\st}\}\cup\{\y\} \label{eq:def_setA}
\end{equation}
as the set of vertices in $\mathscr{G}$ with mark smaller than $T_{\st}$ together with $\y$.

\begin{lemma}\label{lem:bound_ordered_tr}
Let $\x=(x,t), \y=(y,s)\in \R^d\times (0,1)$ with $t>T_{\sp}$ and $s>T_{\sigma_0}$. Then, there exists $\varepsilon>0$ not depending on the choice of $\theta$ such that for $A$ given in \eqref{eq:def_setA} it holds 
\[
\E_{\x,\y}\bigg[\sum_{p\in Q_A\cup R_A} (2\lambda)^{\abs{p}}\, \big \vert\, \x\sim\y \bigg] < \lambda^\varepsilon.
\]  
\end{lemma}
\begin{proof}
We set for the duration of this proof $\ell=T_{\sp}$.
Let $0<\varepsilon<(\frac{1}{\gamma}-1)\wedge \sigma_0$, which is possible as $\frac{1}{\gamma}>1$. For $n\geq 1$, denote by $\hat{Q}_A^n$ the set of paths in $Q_A^n$ which do not visit $\y$ in their last step and denote by $\hat{Q}_\y^n := Q_A^n\backslash \hat{Q}_A^n$ the set of paths in $Q_A^n$ with $\y$ as its last vertex. Note that by definition the paths in $(Q_A\cup R_A)\backslash\hat{Q}_\y^1$ never consist of the given edge between the vertices $\x$ and $\y$. Thus, the occurence of paths in $\mathscr{G}$ which belong to $(Q_A\cup R_A)\backslash\hat{Q}_\y^1$ is independent of the edge between $\x$ and $\y$ by Assumption \ref{ass:main}. Then, by Mecke's equation and Assumption \ref{ass:main} we have, for $n\geq 2$, that
\begin{align*}
\E_{\x,\y}&\big[\vert\hat{Q}_A^n\vert\, \big\vert\,\x\sim\y\big]\\
& = \int_{\mathbb{R}^d\times (\ell,1]}\mathd \x_1 \cdots \int_{\mathbb{R}^d\times (0,\ell]}\mathd \x_n \prod_{i=1}^n \rho\big(\kappa_2^{-1/\delta}(t_i\wedge t_{i-1})^{1-\gamma}(t_i\vee t_{i-1})^{\gamma}\abs{x_i - x_{i-1}}^d\big),
\end{align*}
where $\x_i = (x_i,t_i)$ for $i=1,\ldots,n$ and $\x = (x_0,t_0)$. Integration over the locations on the right-hand side and using \eqref{eq:alp_bound} yields, for $n\geq 2$, that
\begin{align*}
\E_{\x,\y}\big[\vert\hat{Q}_A^n\vert\, \big\vert\,\x\sim\y\big] &\leq I_\rho\int_0^\ell \mathd t_n \nu_{\ell,n}^{t_0}(t_n) = I_\rho\alpha_n\tfrac{\ell^{1-\gamma}}{1-\gamma}\\
&\leq C^n \lambda^{(1/\gamma-2)n}\lambda^2\big(\lambda^{-1/\gamma}\log(\tfrac{1}{\lambda})^{(1-\gamma)/\gamma}+ \lambda^{-2}\log(\tfrac{1}{\lambda})\big)
\end{align*}
for some constant $C>0$, where we have used in the last step that $t_0> T_{\sp}$ and that $T_{\sp}=\frac{\lambda^{2/\gamma}}{\theta^{1/\gamma}}\log(1/\lambda)^{-1/\gamma}$. As it is easy to see that this bound also holds for the case $n=1$, there exists $C>0$ such that, for $\lambda$ small enough, we have
\[
\sum_{n=1}^\infty (2\lambda)^n \E_{\x,\y}\big[\vert\hat{Q}_A^n\vert \,\big\vert\,\x\sim\y\big] \leq C \lambda^{1/\gamma+1}\big(\lambda^{-1/\gamma}\log(\tfrac{1}{\lambda})^{(1-\gamma)/\gamma}+ \lambda^{-2}\log(\tfrac{1}{\lambda})\big) < 2C\lambda^\varepsilon.
\] 
By \eqref{eq:alp_bet_def} there exists $c>0$, such that $\alpha_n > c(\alpha_{n-1}+\beta_{n-1})$ for $n\geq 2$, and therefore similarly to the previous calculation it holds that
\begin{align*}
\E_{\x,\y}\big[\vert R_A^n\vert\,\big\vert\, \x\sim\y\big]  &\leq n\int_\ell^1 \mathd t_{n-1} \nu_{\ell,n-1}^{t_0}(t_{n-1})= c n \alpha_n\\
&\leq c^n \lambda^{(1/\gamma-2)(n-2)}\big(\lambda^{-1/\gamma}\log(\tfrac{1}{\lambda})^{(1-\gamma)/\gamma}+ \lambda^{-2}\log(\tfrac{1}{\lambda})\big),
\end{align*}
where $c>0$ changes throughout the lines and therefore there exists $C>0$ such that, for $\lambda$ small enough, we have
\[
\sum_{n=3}^\infty (2\lambda)^n \E_{\x,\y}\big[\vert R_A^n\vert\,\big\vert\, \x\sim\y\big] \leq  C \lambda^{1/\gamma+1}\big(\lambda^{-1/\gamma}\log(\tfrac{1}{\lambda})^{(1-\gamma)/\gamma}+ \lambda^{-2}\log(\tfrac{1}{\lambda})\big) < 2C\lambda^\varepsilon.
\]
Note that with the same calculation this bound also holds for $\sum_{n=3}^\infty (2\lambda)^n \E\vert \hat{Q}_{\y}^n\vert$. Thus, it is left to find a bound for $2\lambda \E_{\x,\y}\big[\vert \hat{Q}_\y^1\vert\,\big\vert\, \x\sim\y\big]$ and $(2\lambda)^2 \E_{\x,\y}\big[\vert \hat{Q}_\y^2\vert\,\big\vert\, \x\sim\y\big]$. By definition it directly follows that the first term is smaller than $2\lambda<\lambda^\varepsilon$. For the second term note that $\vert\hat{Q}_\y^2\vert$ is dominated by the number of neighbours of $\y$ which are not $\x$. As the expectation of this number is smaller than $\frac{I_\rho}{(1-\gamma)\gamma} T_{\sigma_0}^{-\gamma}$, we have $(2\lambda)^2 \E_{\x,\y}\big[\vert \hat{Q}_\y^2\vert\,\big\vert\, \x\sim\y\big] \leq C \lambda^{\sigma_0}< C\lambda^\varepsilon$
for some $C>0$. 
\end{proof}

Recall that there exists a constant $C>0$ such that for any vertex $\x = (x,t) \in \R^d\times(0,1)$, its expected degree is smaller than $Ct^{-\gamma}$. We now set $\theta$ in the definition of $T_{\sp}$ as $\theta := \frac{\varepsilon}{8Cc}$, where $\varepsilon>0$ is given by Lemma \ref{lem:bound_ordered_tr} and $c>0$ by Lemma \ref{lem:cite_3}. The following result is then a consequence of Lemma \ref{lem:bound_ordered_tr} and follows with the same argumentation as \cite[Proposition 5.8]{LinMSV2021}. It provides the bound \eqref{eq:non_survival_mid_vertices} in the proof of Proposition \ref{prop:contact_upper_bound}.

\begin{lemma}\label{lem:non_survival_mid_vertices}
Let $\x=(x,t), \y=(y,s)\in \R^d\times (0,1)$ with $t>T_{\sp}$ and $s>T_{\sigma_0}$ and $(\xi^\x_t)_{t\geq 0}$ be the contact process on $\mathscr{G}$ with rate $\lambda$ which only starts in $\x$. Then, there exists $\varepsilon>0$ such that
\[
\P_{\x,\y}(\xi^\x_t\neq \emptyset\ \forall t\geq 0\,\vert\, \x\sim\y) \leq \lambda^\varepsilon
\]
when $\lambda$ is small.
\end{lemma}
\begin{proof}
As seen before the expected degree of a vertex with mark larger than $T_{\sp}$ is smaller than $\frac{C\theta}{\lambda^2}\log(\frac{1}{\lambda})$. Thus, to find an upper bound for the survival probability of $(\xi^\x_t)_{t\geq 0}$ we look at whether the degree of $\x$ is smaller than $\frac{2C\theta}{\lambda^2}\log(\frac{1}{\lambda})$ or not. Then, by Lemma \ref{lem:cite_3} with $T= \lambda^{-\varepsilon/2}$ and Lemma \ref{lem:bound_ordered_tr} it holds that
\begin{align*}
\P_{\x,\y}&(\xi^\x_t\neq \emptyset\ \forall t\geq 0\,\vert\, \x\sim\y)\\
&\leq \P_{\x,\y}\big(\deg(\x)>\frac{2C\theta}{\lambda^2}\log(\tfrac{1}{\lambda})\,\vert\, \x\sim\y\big) + \frac{\exp\big(\frac{\varepsilon}{4}\log(\frac{1}{\lambda})\big)}{T}+T \E_{\x,\y}\big[\sum_{p\in Q_A\cup R_A}(2\lambda)^{\abs{p}}\,\vert\, \x\sim\y\big]\\
&\leq \P_{\x,\y}\big(\deg(\x)>\frac{2C\theta}{\lambda^2}\log(\tfrac{1}{\lambda})\,\vert\, \x\sim\y\big) + \lambda^{\varepsilon/4} + \lambda^{\varepsilon/2}.
\end{align*}
As the number of neighbours of $\x$ different to $\y$ is Poisson distributed with parameter at most $\frac{C\theta}{\lambda^2}\log(\frac{1}{\lambda})$, using a Chernoff bound yields
\[
\P_{\x,\y}\big(\deg(\x)>\frac{2C\theta}{\lambda^2}\log(\tfrac{1}{\lambda})\,\vert\, \x\sim\y\big) \leq \exp\big(-c_1 \frac{\theta}{\lambda^2}\log(\tfrac{1}{\lambda})\big) < \lambda
\]
for $\lambda$ small enough, where $c_1>0$ is some constant. Thus, $\P_{\x,\y}(\xi^\x_t\neq \emptyset\ \forall t\geq 0\,\vert\, \x\sim\y) \leq 3\lambda^{\varepsilon/4}$ which completes the proof.
\end{proof}

As the last step to complete the proof of Proposition \ref{prop:contact_upper_bound} we show inequality \eqref{eq:weak_paths_not_inf}. Recall that, for $\sigma >0$, $E_\sigma$ denotes the event that each infection path of the contact process $(\xi^{\Palm}_t)_{t\geq 0}$ which jumps at first to a vertex with mark larger than $T_\sigma$ is finite and never reaches a vertex with mark smaller than $T_\sigma$.
\begin{lemma}\label{lem:weak_paths_not_inf}
There exists $\varepsilon>0$ and $\sigma>0$, such that
\[
\P_{\Palm}(E_\sigma^c\cap \set{T_0\geq T_\sigma}) \leq \lambda^{2/\gamma-1+\varepsilon}.
\]
\end{lemma}
\begin{proof}
We denote by
\[
B_\sigma = \set{\x \in \X: \x\neq \0, t\leq T_\sigma}
\] 
the set of vertices with mark smaller or equal to $T_\sigma$ and by $Q_{B_\sigma}^n$ and $R^n_{B_\sigma}$ the associated sets of paths, which either visit a vertex in $B_\sigma$ in its last step or whose vertices are not in $B_\sigma$ but the last vertex is equal to a previous one. We set 
\[Q_{B_\sigma} = \bigcup_{n\geq 2} Q_{B_\sigma}^n,\quad R_{B_\sigma} = \bigcup_{n\geq 3} R^n_{B_\sigma}
\] and $P_0 = \big\{\big((0,T_0),\x,(0,T_0),\y\big):\x,\y \sim (0,T_0)\big\}$. 
By \cite[Lemma 5.10]{LinMSV2021} the event that no infection path starting in $(0,T_0)$ has an ordered trace in $P_0\cup Q_{B_\sigma}\cup R_{B_\sigma}$ implies $E_\sigma$. In fact, if no infection path starting in $(0,T_0)$ has ordered trace in $P_0\cup Q_{B_\sigma}\cup R_{B_\sigma}$, then each infection path $g:I\to V$ which starts at $(0,T_0)$ and jumps to a vertex $\x=(x,t)$ with $t\geq T_\sigma$ never visits a vertex in $B_\sigma$ and never visits a vertex outside of $B_\sigma$ more than once. Thus, by Lemma \ref{lem:cite_2} any such infection path is finite, as the contact process starting in $(0,T_0)$ and restricted to $B_\sigma^c$ is thin outside $(0,T_0)$ and dies out. Hence, it then holds
\begin{equation}
\begin{split}
\P&_{\Palm}(E_\sigma^c\cap \set{t_0\geq T_\sigma})\\
&\leq \E_{\Palm}\big[1_{\{t_0\geq T_\sigma\}} \sum_{p\in P_0\cup Q_{B_\sigma}\cup R_{B_\sigma}} (2\lambda)^{\abs{p}}\big]\\
&\leq(2\lambda)^3 \E_{\Palm}[\abs{P_0}1_{\{t_0>T_\sigma\}}] + \sum_{n=2}^\infty (2\lambda)^n \E_{\Palm}\abs{Q_{B_\sigma}^n} + \sum_{n=3}^\infty (2\lambda)^n \E_{\Palm}\abs{R_{B_\sigma}^n}.
\end{split} \label{eq:ordered_tr_count2}
\end{equation}
We will proceed to find upper bounds for the expected number of ordered traces corresponding to each of the three classes $P_0, Q_{B_\sigma}$ and $R_{B_\sigma}$. First, it holds by Mecke's equation and integration over the location of the vertices that
\begin{align*}
(2\lambda)^3 &\E_{\Palm}[\abs{P_0}1\{T_0>T_\sigma\}]\\
&\leq (2\lambda)^3 I_\rho^2\int_{T_\sigma}^1\mathd t_0 \int_0^1 \mathd s \int_0^1 \mathd t (t_0\wedge s)^{-\gamma}(t_0\wedge s)^{\gamma-1}(t_0\wedge t)^{-\gamma}(t_0\wedge t)^{\gamma-1}\\
&\leq (2\lambda)^3 (CI_\rho)^2\int_{T_\sigma}^1\mathd t_0 t_0^{-2\gamma}\\
&\leq C^3\lambda^{2/\gamma - 1 + \sigma(2-1/\gamma)} < C^3\lambda^{2/\gamma-1+\varepsilon},
\end{align*}
for some small $\varepsilon>0$ as $2/\gamma - 1 + \sigma(2-1/\gamma)$ is increasing in $\sigma$ as $\gamma>\frac{1}{2}$. The positive constant $C$ does not depend on $\lambda$ and $\sigma$ but may change throughout the lines.

Using Lemma \ref{lem:nu_bound} and \eqref{eq:alp_bound} with $\ell = T_\sigma = \lambda^{2-\sigma}$ as done in the proof of Lemma \ref{lem:bound_ordered_tr} we have by Mecke's equation and Assumption \ref{ass:main}, for $n \geq 2$, that
\begin{align*}
\E_{\Palm}\abs{Q_{B_\sigma}^n} &= \int_0^1 \mathd t_0 \int_0^\ell \mathd t_n \nu_{\ell,n}^{t_0}(t_n)\\
&\leq \frac{1}{1-\gamma}(2c)^{n-2}c^2\ell^{(1-2\gamma)(n/2-1)}\ell^{1-\gamma} \int_0^1 \mathd t_0 (\ell^{1/2-\gamma}t_0^{1-\gamma} + t_0^{-\gamma})\\
&\leq C^n \lambda^{(1/\gamma-2)(n/2-1)(2-\sigma)}\lambda^{(1/\gamma-1)(2-\sigma)}\lambda^{(1/(2\gamma)-1)(2-\sigma)},
\end{align*}
for some positive constant $C>0$. Then, it follows that
\begin{align*}
\sum_{n=2}^\infty (2\lambda)^n \E_{\Palm}\abs{Q_{B_\sigma}^n} &\leq C^2 \lambda^{3/\gamma-1}\lambda^{(2-1/\gamma -1/(2\gamma))\sigma} < C^2\lambda^{2/\gamma - 1 + \varepsilon}
\end{align*}
for $\sigma>0$ sufficiently small. For the last summand Lemma \ref{lem:nu_bound} and \eqref{eq:alp_bound} yield similarly that
\begin{align*}
\E_{\Palm}\abs{R_{B_\sigma}^n}&\leq n \int_0^1 \mathd t_0 \int_\ell^1\mathd t_{n-1} \nu_{\ell,n-1}^{t_0}(t_{n-1})\\
&\leq \frac{n}{c(1-\gamma)\gamma} \int_0^1 \mathd t_0 \alpha_n\\
&\leq \frac{n}{(1-\gamma)\gamma} (2c)^{n-2}c \ell^{(1-2\gamma)(n/2-1)} \int_0^1 \mathd t_0 (\ell^{1/2-\gamma}t_0^{1-\gamma} + t_0^{-\gamma}).
\end{align*}
and we have therefore
\begin{align*}
\sum_{n=3}^\infty (2\lambda)^n \E_{\Palm}\abs{R_{B_\sigma}^n} &\leq \sum_{n=3}^\infty C^n \lambda^{(1/\gamma-1)n} \lambda^{2-1/\gamma} \lambda^{(2-1/\gamma)(n/2-1/2)\sigma)}\\
&\leq C^3 \lambda^{2/\gamma-1}\lambda^{(2-1/\gamma)\sigma}<\lambda^{2/\gamma-1+\varepsilon},
\end{align*}
since $2/\gamma-1+\sigma(2-1/\gamma)$ is again increasing in $\sigma$.
As all three summands of the righthand side of \ref{eq:ordered_tr_count2} are bounded by $C\lambda^{2/\gamma-1+\varepsilon}$ for some constants $C>0$ and $\varepsilon>0$ sufficiently small this completes the proof.
\end{proof}

\section{Exponential extinction time on finite restrictions}\label{sec:exp_ext_time}
In this section we consider the graph sequence $(\mathscr{G}_n)_{n\in \N}$, where $\mathscr{G}_n$ is the spatial restriction of $\mathscr{G}$ on $[-\frac{n^{1/d}}{2},\frac{n^{1/d}}{2}]^d$. As this is a sequence of finite graphs, the contact process with any potential initial condition will almost surely die out for all $n\in \N$. Thus, the more natural question is to estimate the time the infection survives on these graphs when the infection starts with the best possible initial condition, i.e. when the graph is fully infected. We denote by $\tau_{n} := \inf\{t>0:\xi_t^{\mathscr{G}_n}=\emptyset\}$ the extinction time of the contact process on $\mathscr{G}_n$. The main result of this section shows that for any choice of $\lambda>0$ the extinction time is at least of exponential order in the number of vertices of $\mathscr{G}_n$ with high probability as $n$ becomes large.

\begin{theorem}\label{thm:exp_extinction_time}
Let $(\mathscr{G}_n)_{n\in\N}$ be the restricted finite graph sequence of a general geometric random graph which satisfies Assumption \ref{ass:main} for $\gamma>\frac{\delta}{\delta+1}$. For any $\lambda>0$, there exists $c>0$ such that
\[
\lim_{n\to \infty} \P\{\tau_{n} \geq e^{cn}\} = 1.
\]
\end{theorem}

\begin{remark}
Note that the result of Theorem \ref{thm:exp_extinction_time} also holds when we consider a graph sequence $(\mathscr{G}_n)_{n\in\N}$, where each graph $\mathscr{G}_n$ is defined on a Poisson process of unit intensity on the torus $\mathbb{T}_n$ with volume $n$ and satisfies Assumption \ref{ass:main}, where the Euclidean metric is replaced by the torus metric.
\end{remark}

As seen in the proof of Proposition \ref{prop:contact_lower_bound} the infection survives well on the neighbourhood of sufficiently powerful vertices, the so called stars. Our main contribution is to show that there exists a connected subgraph in $\mathscr{G}_n$ which contains of order $n$ stars. Then, with similar arguments as done in \cite{MouMVY2016}, it can be shown that the infection survives on this subgraph at least for a time of exponential order in $n$.

\begin{prop}\label{pro:clusterofstars}
Let $S>0$ be given and $(\mathscr{G}_n)_{n\in\N}$ the restricted finite graph sequence of a general geometric random graph which satisfies Assumption \ref{ass:main} for $\gamma>\frac{\delta}{\delta+1}$. Then, there exists $b>0$ and $\varepsilon>0$ such that, for $n$ sufficiently large, the probability that $\mathscr{G}_n$ has a connected subgraph containing $b\cdot n$ disjoint stars of at least $S$ vertices each is larger than $1-\exp(-n^\varepsilon)$.
\end{prop}

\begin{proof}
We fix $0<a<\frac{1}{\log 2}$ and choose $\varepsilon_1>0$ small enough that $\log 2 > \frac{\varepsilon_1+\log 2}{\gamma+\gamma/\delta}$, which is possible since $\gamma>\frac{\delta}{\delta+1}$. 
Similarly to the proof of Lemma \ref{lem:chainofstars}, the vertices with mark smaller than $\frac{1}{2}$ will represent the potential midpoints of the stars of the subgraph, whereas the vertices with larger mark represent potential neighbours and connectors of the midpoints.
For this proof it is not sufficient to use the arguments of Lemma \ref{lem:chainofstars} to show that a line of stars exists in $\mathscr{G}_n$, as such a subgraph would only consist of order $\log(n)$ many stars. Hence, we need to to break up the powerful vertices of $\mathscr{G}_n$ with mark smaller than $\frac{1}{2}$ more carefully into a system of boxes such that each of these boxes contains a midpoint of one potential star. Let $n_p = \gaus{n^{(1-a\log 2)/d}}$ and $k_p = \gaus{(a\log n)/d}$. For $k=0,\ldots k_p$, we define 
\[
V_k := \set{0,\ldots,n_p2^{k_p-k}-1}^d.
\]
and
\[
A_{k,\v} := \bigtimes_{i=1}^d \big(2^k v_i,2^k (v_i+1)\big) \quad \text{for}\ k=0,\ldots ,k_p\ \text{and}\ \v = (v_1,\ldots,v_d) \in V_k.
\]
For each $k= 0,\ldots,k_p$, the cubes $\{A_{k,\v}:\v\in V_k\}$ give a tessellation of $[0,\frac{n^{1/d}}{2}]^d$ into $(n_p2^{k_p-k})^d$ cubes of volume $2^{kd}$ such that the finest tessellation is given for $k=0$ and the coarsest for $k=k_p$. Furthermore, the cubes are nested in each other in the sense that for each cube $A_{k+1,\v}$ the cubes $\{A_{k,2\v+\e}:\e\in \{0,1\}^d\}$ are a tessellation of $A_{k+1,\v}$.

Set $\theta>0$ such that $\log 2>\theta>\frac{\varepsilon_1+\log 2}{\gamma+\gamma/\delta}$ and define 
\[
B_{k,\v} := A_{k,\v}\times (\tfrac{1}{2}e^{-(k+1)\theta d},\tfrac{1}{2}e^{-k\theta d}) \quad \text{for}\ k=0,\ldots k_p\ \text{and}\ \v = (v_1,\ldots,v_d) \in V_k.
\]
We denote with the parameter $k=0,\ldots k_p$ the layer of the boxes $\{B_{k,\v}: \v\in \V_k\}$ which defines the range of the marks of points of $\X$ inside the boxes and the level of coarsness of the tessellation of the space. Thus, large values of $k$ imply more powerful vertices and a coarser set of boxes to separate them. As an example, the boxes of the most powerful layer $k_p$ have width of order $n^{a\log 2}$ and the marks of the vertices therein are of order $n^{-a \theta\log 2}$. 
As we have already seen that the cubes $\{A_{k,\v}:k=0,\ldots,k_p, \v\in V_k\}$ are nested in each other, the system of boxes $\{B_{k,\v}:k=0,\ldots,k_p, \v\in V_k\}$ can be made to have a tree structure by treating $B_{k+1,\v}$ as the parent of each box $B_{k,2\v+\e}$, for $\e\in \{0,1\}^d$, see Figure \ref{fig:box_nested}. This leads to $n_p^d$ distinct $2d$-regular trees with roots $\{B_{k_p,\v}:\v\in V_{k_p}\}$.

\begin{figure}[h]
\begin{center}
\begin{tikzpicture}[scale=0.4, every node/.style={scale=0.6}]
\draw (0,4.5)  -- (0,4.5) node[left] {$\frac{1}{2}$} --  (0, -3) node[below left] {$0$};
\draw (0,-3)  -- (24.5,-3);
%\draw[dashed] (0,-1.25) rectangle ++(16,0.75);
%\draw[dashed] (16,-1.25) -- (24,-1.25);
\node at (4,-1) {$\vdots$}; 
\node at (12,-1) {$\vdots$}; 
\node at (20,-1) {$\vdots$}; 
\draw (0,-0.5) rectangle ++(8,1); \node at (4,-0.1) {$B_{3,1}$};
\draw (8,-0.5) rectangle ++(8,1); \node at (12,-0.1) {$B_{3,2}$};
\draw (16,-0.5) rectangle ++(8,1); \node at (20,-0.1) {$B_{3,3}$};
\draw (0,0.5) rectangle ++(4,1.5); \node at (2,1) {$B_{2,1}$};
\draw (4,0.5) rectangle ++(4,1.5); \node at (6,1) {$B_{2,2}$};
\draw (8,0.5) rectangle ++(4,1.5); \node at (10,1) {$B_{2,3}$};
\draw (12,0.5) rectangle ++(4,1.5); \node at (14,1) {$B_{2,4}$};
\draw (16,0.5) rectangle ++(4,1.5); \node at (18,1) {$B_{2,5}$};
\draw (20,0.5) rectangle ++(4,1.5); \node at (22,1) {$B_{2,6}$};
\draw (0,2) rectangle ++(2,2.5); \node at (1,3) {$B_{1,1}$};
\draw (2,2) rectangle ++(2,2.5); \node at (3,3) {$B_{1,2}$};
\draw (4,2) rectangle ++(2,2.5); \node at (5,3) {$B_{1,3}$};
\draw (6,2) rectangle ++(2,2.5); \node at (7,3) {$B_{1,4}$};
\draw (8,2) rectangle ++(2,2.5); \node at (9,3) {$B_{1,5}$};
\draw (10,2) rectangle ++(2,2.5); \node at (11,3) {$B_{1,6}$};
\draw (12,2) rectangle ++(2,2.5); \node at (13,3) {$B_{1,7}$};
\draw (14,2) rectangle ++(2,2.5); \node at (15,3) {$B_{1,8}$};
\draw (16,2) rectangle ++(2,2.5); \node at (17,3) {$B_{1,9}$};
\draw (18,2) rectangle ++(2,2.5); \node at (19,3) {$B_{1,10}$};
\draw (20,2) rectangle ++(2,2.5); \node at (21,3) {$B_{1,11}$};
\draw (22,2) rectangle ++(2,2.5); \node at (23,3) {$B_{1,12}$};
\end{tikzpicture}
\caption{Sketch of the structure of the boxes $B_{k,\v}$ in dimension one. The y-axis represents the mark of the vertices and the x-axis the location.}\label{fig:box_nested}
\end{center}
\end{figure}
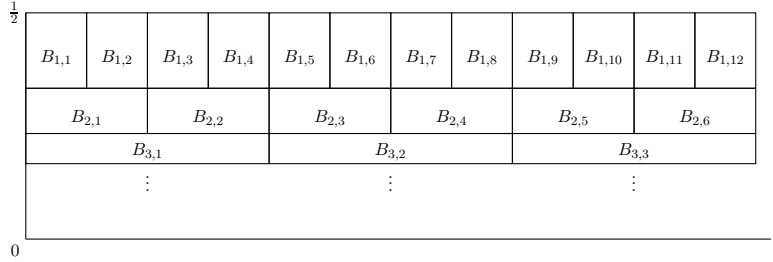

Note that the amount of vertices in a box $B_{k,\v}$ is Poisson-distributed with parameter 
\[
2^{kd}\big(\tfrac{1}{2}e^{-k\theta d}-\tfrac{1}{2}e^{-(k+1)\theta d}\big)> ce^{kd(\log 2 - \theta)}
\]
for some constant $c>0$ not depending on $k$ and $\v$. Thus, for $\varepsilon_2 = \log 2 - \theta>0$ it holds that
\begin{equation}
\P\{B_{k,\v}\ \text{is non-empty}\}\geq 1 - \exp\big(ce^{kd\varepsilon_2}\big). \label{eq:box-non-empty}
\end{equation}
On the event that $B_{k,\v}$ is non-empty we denote by $\x_{k,\v}$ the vertex with the smallest mark in the box.

As mentioned above each of these boxes corresponds to one potential midpoint of the stars. For each such midpoint, i.e. for each such box we need a distinct set of potential neighbours and connectors. To this end, we colour the vertices with mark larger than $\frac{1}{2}$. Choose $0<\varepsilon_3 < \theta\gamma \wedge \delta \varepsilon_1$ such that $\sum_{k=0}^\infty e^{-kd(\theta\gamma\wedge \delta\varepsilon_1 - \varepsilon_3)}$ converges and color the points of $\X$ on $\R^d\times[\frac{1}{2},1)$ by the colour set $\N$ independently such that the points with color $k\in \N$ form a Poisson point process $\X_k$ on $\R^d\times[\frac{1}{2},1)$ with an intensity proportional to $e^{-kd(\theta\gamma\wedge \delta\varepsilon_1 - \varepsilon_3)}$.
For $k=0,\ldots,k_p$ and $\v\in V_k$, we denote by
\begin{itemize}
\item $N_{k,\v}=\X_k\cap A_{k,\v}\times [\frac{1}{2},\frac{3}{4})$ the potential neighbours of $\x_{k,\v}$, if $B_{k,\v}$ is non-empty and by
\item $C_{k,\v}=\X_k\cap A_{k,\v}\times [\frac{3}{4},1)$ the potential connectors, which we will use to connect $\x_{k,\v}$ to other midpoints.
\end{itemize} 
Note that the intensity of these Poisson point processes is decreasing in $k$, as it is easier for vertices with smaller mark to find sufficiently many neighbours and connectors in the corresponding boxes, so we require fewer candidates to succeed.

Since $\theta\gamma<\log 2$, there exists $c>0$ such that, for all $k=0,\ldots,k_p$, on the event that $B_{k,\v}$ is non-empty it holds $ct_{k,\v}^{-\gamma/d} <2^{k}$, where $t_{k,\v}$ is the mark of $\x_{k,\v}$. Thus for each $k=0,\ldots,k_p$, the volume of $B(\x_{k,\v},ct_{k,\v}^{-\gamma/d}) \cap A_{k,\v}$ is a positive proportion $\rho>\frac{1}{2^d}$ of the volume of the ball itself. As in the proof of Lemma \ref{lem:chainofstars} this leads to two observations.
First, given $B_{k,\v}$ is non-empty, by the same arguments as in the proof of Lemma \ref{lem:twoconn} the number of neighbours of $x_{k,\v}$ in $N_{k,\v}$ is Poisson-distributed with parameter larger than $ce^{-kd(\theta\gamma\wedge \delta\varepsilon_1-\varepsilon_3)}e^{kd\theta\gamma} > ce^{kd\varepsilon_3}$ for some constant $c>0$ not depending on $k$. We denote by $\St (k,\v)$ the event that $B_{k,\v}$ is non-empty and $\x_{k,\v}$ has at least $S$ neighbours in $N_{k,\v}$. Then, by a Chernoff bound there exists $c>0$ and $k_0$ sufficiently large and not depending on $n$ such that for all $k\geq k_0$ and $\v\in V_k$ it holds
\begin{equation}
\P\big(\St (k,\v)\,\vert\, B_{k,\v}\ \text{is non-empty}\big) \geq 1 - \exp(-ce^{kd\varepsilon_3}). \label{eq:box-star}
\end{equation}

Second, given the box $B_{k+1,\v}$ and one of its children $B_{k,2\v+\e}$ are non-empty, note that $\x_{k+1,\v}$ and $\x_{k,2\v+\e}$ have distance at most $\sqrt{d}2^{k+1}$ and both have marks smaller $\frac{1}{2}e^{-k\theta d}$. Thus, by the same argument as used in the proof of Lemma 2.3, there exists a constant $c>0$ such that the number of vertices in $C_{k,2\v+\e}$ which form an edge to both $\x_{k+1,\v}$ and $\x_{k,2\v+\e}$ is Poisson-distributed with parameter larger than
\begin{align*}
ce^{-kd(\theta\gamma\wedge \delta\varepsilon_1-\varepsilon_3)}t_{k,2\v+\e}^{-\gamma}\big(1\wedge t_{k+1,\v}^{-\gamma\delta}(\abs{x_{k,2\v+\e}-x_{k+1,\v}}+t_{k,2\v+\e}^{\gamma/d})^{-d\delta}\big) > ce^{kd\varepsilon_3}
\end{align*}
where the constant $c>0$ changes through the steps but does not depend on $k$ and $\v$ and we have used that $\theta>\frac{\varepsilon_1+\log 2}{\gamma+\gamma/\delta}$. Denote by $\x_{k+1,\v} \overset{2}{\conn} \x_{k,2\v+\e}$ the event that the boxes $B_{k+1,\v}$ and $B_{k,2\v+\e}$ are non-empty and the vertices $\x_{k+1,\v}$ and $\x_{k,2\v+\e}$ are connected via a vertex in $C_{k,2\v+\e}$. Then, we have for all $k\in \N$ and $\v\in V_k$ that
\begin{equation}
\P(\x_{k+1,\v} \overset{2}{\conn} \x_{k,2\v+\e}\,\vert\,  B_{k+1,\v}\ \text{and}\ B_{k,2\v+\e}\ \text{are non-empty}) \geq 1 - \exp(-ce^{kd\varepsilon_3}).\label{eq:box-twoconn}
\end{equation}

With the structure of the boxes and the bounds \eqref{eq:box-non-empty}-\eqref{eq:box-twoconn} at hand we will show that there exists a connected subgraph containing of order $n$ distinct stars with at least $S$ vertices each. This will be done in two steps. First, we show that the vertices in the most powerful layer $k_p$ form a connected subgraph containing $n_p^d$ distinct stars, where each box $B_{k_p,\v}$ contains one of the midpoints of these stars. Second, recall that each box $B_{k_p,\v}$ for $\v\in V_{k_p}$ represents the root of a $2^d$-regular tree. We will show that the trees resulting only from boxes contributing a star to the connected subgraph are percolated $2^d$-regular trees with a depth of order $k_p$ containing of order $2^{k_p d}$ distinct stars. As there are $n_p^d$ many trees like this, this will lead to a connected subgraph of $\mathscr{G}_n$ with of order $n$ distinct stars.
  
To simplify notation we redefine the labeling of the boxes of layer $k_p$. Let 
\[
\sigma: \{0,\ldots,n_p^d-1\}\to V_k
\]
be a bijection such that $B_{k_p,\sigma(0)}=B_{k_p,\0}$ and the boxes $B_{k_p,\sigma(i)},\ B_{k_p,\sigma(i+1)}$ are adjacent to each other for $i=0,\ldots,n_p^d-2$. In the same way we relabel the vertices with the smallest mark in a box, i.e. on the event that $B_{k_p,\sigma(i)}$ is non-empty we denote by $\x_{k_p,\sigma(i)}$ the vertex with the smallest mark in that box. We say $B_{k_p,\sigma(0)}$ is good if and only if the box is non-empty and the vertex $\x_{k_p,\sigma(0)}$ has at least $S$ neighbours in $N_{k_p,\sigma(0)}$. By \eqref{eq:box-non-empty} and \eqref{eq:box-star}, there exists $c>0$ such that 
\[
\P(B_{k_p,\sigma(0)}\ \text{is good}) \geq \big(1- \exp(-ce^{k_p d\varepsilon_2})\big)\big(1- \exp(-ce^{k_p d\varepsilon_3})\big).
\]
For $i=0,\ldots,n_p^d-2$, we say $B_{k_p,\sigma(i+1)}$ is good if
\begin{enumerate}[(i)]
\item $B_{k_p,\sigma(i)}$ is good,
\item $B_{k_p,\sigma(i+1)}$ is non-empty,
\item $\x_{k_p,\sigma(i+1)}$ has at least $S$ neighbours in $N_{k_p,\sigma(i+1)}$ and
\item $\x_{k_p,\sigma(i+1)}$ and $\x_{k_p,\sigma(i)}$ are connected via a connector in $C_{k_p,\sigma(i+1)}$
\end{enumerate}
and otherwise bad. Let $\varepsilon_4< \varepsilon_2\wedge \varepsilon_3$. Then, by \eqref{eq:box-non-empty}-\eqref{eq:box-twoconn}, there exists $c>0$ such that, for $i=0,\ldots,n_p^d-2$, it holds
\[
\P(B_{k_p,\sigma(i+1)}\ \text{is good}\,\vert\, B_{k_p,\sigma(i)}\ \text{is good}) \geq \big(1- \exp(-ce^{kd\varepsilon_4})\big).
\]
Thus we can deduce that
\begin{equation}
\P(B_{k_p,\v}\ \text{is good for all}\ \v \in V_{k_p})\geq 1-n_0^d\exp(-ce^{k_pd\varepsilon_4}) \geq 1-n^{1-a\log 2}\exp(-cn^{a\varepsilon_4}).\label{eq:goodpowerfullayer}
\end{equation}

We continue the definition of good boxes on the other layers. For $\v\in \mathbb{\N^d}$ denote by $\gaus{\frac{\v}{2}}$ the vector $(\gaus{\frac{v_1}{2}},\ldots,\gaus{\frac{v_d}{2}})$. Then, for each $k=0,\ldots,k_p-1$ and $\v\in V_k$ the parent box of $B_{k,\v}$ is given by $B_{k+1,\gaus{\frac{\v}{2}}}$. For $k=0,\ldots,k_p-1$ and $\v\in V_k$, we say that $B_{k,\v}$ is good if
\begin{enumerate}[(i)]
\item $B_{k+1,\gaus{\frac{\v}{2}}}$ is good,
\item $B_{k,\v}$ is non-empty,
\item $\x_{k,\v}$ has at least $S$ neighbours in $N_{k,\v}$ and
\item $\x_{k,\v}$ and $\x_{k+1,\gaus{\frac{\v}{2}}}$ are connected via a connector in $C_{k,\v}$
\end{enumerate}
and otherwise we say that $B_{k,\v}$ is bad. Note that again \eqref{eq:box-non-empty}- \eqref{eq:box-twoconn} implies that there exists $c>0$ such that for all $k=0,\ldots,k_p-1$, $\v\in V_k$ and $\e\in \{0,1\}^d$, it holds
\begin{equation}
\P(B_{k,2\v+\e}\ \text{is good}\,\vert\, B_{k+1,\v}\ \text{is good}\} \geq 1-\exp(-ce^{kd\varepsilon_4})\label{eq:good_children}
\end{equation}
and given $B_{k+1,\v}$ is good, the events that $B_{k,2\v+\e}$ is good are independent of each other and any other box on this layer, since they depend on disjoint subsets of $\mathcal{X}$ and edges occur independently. Therefore, the number of good children of $B_{k+1,\v}$, given this box is good, is Binomial-distributed with parameters $2^d$ and $p_k>1-\exp(-ce^{kd\varepsilon_4})$ and, given the good boxes on layer $k+1$, the numbers of good children of each of those good boxes are independent of each other. Consequently, denote by $\abs{B_k}$ the number of good boxes in layer $k$. Then, given $\abs{B_{k+1}}$, $\abs{B_k}$ is Binomial-distributed with parameters $2^d\abs{B_{k+1}}$ and $p_k$.

With this observation we are able to give estimates on the number of good boxes in each layer and show that sufficiently many good boxes exists. For $k=0,\ldots,k_p-1$, we denote by $E_k := \set{\abs{B_k}>2^d(1-k^{-2})\abs{B_{k+1}}}$ the event that layer $k$ has sufficiently many good boxes in comparison to layer $k+1$ of the next more powerful vertices and we denote by $E_{k_p}$ the event that all boxes in layer $k_p$ are good. Then, the event $E_k\cap\ldots\cap E_{k_p}$ implies that
\[
\abs{B_k} > \abs{B_{k_p}} \prod_{i=k}^{k_p-1} 2^d(1-i^{-2}) > c n_0^d 2^{d(k_p-k-1)} > c2^{-kd}n,
\]
where $c=\prod_{i=1}^\infty (1-i^{-2})$. Thus, it is sufficient to show that there exists $k_0$ such that $E_{k_0}\cap \ldots \cap E_{k_p}$ holds with high probability. We choose $k_0$ large enough that \eqref{eq:box-star} still holds for all $k\geq k_0$ and $\v\in V_k$ and that $1-\exp(ce^{k_0d\varepsilon_4})>1-k_0^{-2}$. Then, by a Chernoff bound for Binomial-distributed random variables it holds
\[
P\big(E_k^c\,\big\vert\, \abs{B_{k+1}}\big) \leq \exp\big(-\frac{2^{d-1}\abs{B_{k+1}}ce^{kd\varepsilon_4}}{k^2}\big)
\]
for all $k\geq k_0$. As a consequence, there exists $c>0$ such that
\[
\P(E_k\,\vert\, E_{k+1}\cap \ldots\cap E_{k_p}) \geq 1-\exp(-c\frac{e^{kd\varepsilon_4} 2^{-kd}n^d}{k^2})\geq 1-\exp\big(-c\frac{n^{ad\varepsilon_4}}{(\log n)^2}\big).
\]
Hence, it follows together with \eqref{eq:goodpowerfullayer} that
\[
\P(E_{k_0}\cap\ldots \cap E_{k_p}) \geq \P(E_{k_p}) \bigg(1-\gaus{a\log n}\exp\big(-c\frac{n^{ad\varepsilon_4}}{(\log n)^2}\big)\bigg) \geq 1-\exp(-n^\varepsilon)
\]
for some $\varepsilon$ not depending on $n$. As $E_{k_0}\cap \ldots\cap E_{k_p}$ implies the existence of up to a constant at least $2^{-k_0}n$ good boxes and therefore the existence of a connected subgraph of $\mathscr{G}_n$ containing $bn$ distinct stars, for some $b>0$, this completes the proof.
\end{proof}

\begin{proof}[Proof of Theorem \ref{thm:exp_extinction_time}]
Given the subgraph $G_n = (V_n,E_n)$ provided by Proposition \ref{pro:clusterofstars} note that $G_n$ is a connected tree by construction. Denote by $M_n\subset V_n$ the set of vertices which are the midpoints of the stars containing $S$ vertices. For $\x,\y \in M_n$, write $\x \overset{2}{\leftrightarrow}\y$ if there exists a connector in $G_n$ which forms an edge to $\x$ and $\y$. By definition all vertices $\x,\y \in M_n$ with $\x \overset{2}{\leftrightarrow}\y$ have graph distance at most two in $G_n$. The graph $H_n$ given by the vertex set $M_n$ and the edge set
\[
F_n := \set{\{\x,\y\}: \x,\y\in M_n, \x \overset{2}{\leftrightarrow}\y}
\]
is a connected tree with degree bounded by $2^d+2$ and for each pair $\x,\y \in M_n$ with $\x \overset{2}{\leftrightarrow}\y$ the connector is unique. Hence, for any $\lambda>0$ and with $S>0$ chosen sufficiently large depending on $\lambda$ by the same arguments as used in the proof of \cite[Theorem 1.4]{MouMVY2016} together with \cite[Proposition 5.2]{MouMVY2016} it holds 
$\lim_{n\to \infty} \P\{\tau_{n} \geq e^{cn}\} = 1$ for some constant $c>0$.
\end{proof}

{\bf Acknowledgment:} This research was supported by Deutsche Forschungsgemeinschaft (DFG) as Project Number~425842117. The authors would like to thank Amitai Linker and Peter M\"orters for the helpful discussions resulting in this project. This work forms part of the second author's PhD thesis.

\printbibliography
\end{document}